\documentclass[11pt]{article}
\usepackage[a4paper]{geometry}
\usepackage{amsthm}
\usepackage{amsmath}
\usepackage{amssymb}
\usepackage{amsfonts}
\usepackage{epsfig,graphicx}
\usepackage{url}
\usepackage{subfigure}

\def\R{\mathbb{R}}

\begin{document}

\theoremstyle{theorem}
\newtheorem{Lemma}{Lemma}[section]
\newtheorem{Proposition}{Proposition}[section]
\newtheorem{Corollary}{Corollary}[section]
\newtheorem{Theorem}{Theorem}[section]
\newtheorem{Condition}{Condition}[section]
\newtheorem{Integrator}{Integrator}

\theoremstyle{definition}
\newtheorem{Definition}{Definition}[section]
\newtheorem{Remark}{Remark}[section]
\newtheorem{Example}{Example}[section]

\newcommand{\todo}[1]{\vspace{5 mm}\par \noindent
\marginpar{\textsc{ToDo}}
\framebox{\begin{minipage}[c]{0.95 \textwidth}
\tt #1 \end{minipage}}\vspace{5 mm}\par}

\title{Variational and linearly-implicit integrators, with applications}

\author{Molei Tao and Houman Owhadi}

\date{August 25, 2014}

\maketitle

\begin{abstract}
We show that symplectic and linearly-implicit integrators proposed by Zhang and Skeel \cite{ZhSk97} are variational linearizations of Newmark methods. When used in conjunction with penalty methods (i.e., methods that replace constraints by stiff potentials), these integrators permit coarse time-stepping of holonomically constrained mechanical systems and bypass the resolution of nonlinear systems. Although penalty methods are widely employed, an explicit link to Lagrange multiplier approaches appears to be lacking; such a link is now provided (in the context of two-scale flow convergence \cite{FLAVOR10}). The variational formulation also allows efficient simulations of mechanical systems on Lie groups.
\end{abstract}

\section{Introduction and main results}
\label{SectionIntroduction}
\paragraph{Integrators:}
Symplectic integrators are popular for simulating mechanical systems due to their structure preserving properties (e.g., \cite{Hairer06}). Implicit methods, on the other hand, allow accurate coarse time-stepping of a class of stiff or multiscale problems (e.g., \cite{LiAbE:08, FiJi10}). It is also a classical treatment to linearize implicit methods so that expensive nonlinear solves can be avoided (e.g., \cite{BeWa76}). Although linearizations of most implicit symplectic methods are not symplectic, Zhang and Skeel found a family of symplectic and linearly-implicit integrators \cite{ZhSk97}, which allows efficient and structure preserving simulations. We show that their method is not only symplectic but in fact variational.

Specifically, consider mechanical systems governed by Newton's equation:
\begin{equation}
    \dot{x}= v, \quad M\dot{v}=-\nabla V(x),
    \label{NewtonEq}
\end{equation}
where $V\in \mathcal{C}^2(\R^n)$ and $M$ is a $n\times n$ symmetric, positive-definite constant matrix.

If we consider the following discrete Lagrangian (see Section \ref{Sectionderivation} for explanations):
\begin{align}
    & \tilde{\mathcal{L}}_d(x_k, x_{k+1}, a_k, a_{k+1})=h \Big( \frac{1}{2} \big( \frac{x_{k+1}-x_k}{h} \big)^T M \big( \frac{x_{k+1}-x_k}{h} \big) \nonumber\\
    & - \frac{1}{2} \big( \beta h^2 \frac{1}{2} a_k^T M a_k + V(x_k) + \beta h^2 a_k^T \nabla V(x_k) + \frac{1}{2} \beta^2 h^4 a_k^T \text{Hess} V(x_k) a_k \big) \nonumber\\
    & - \frac{1}{2} \big( \beta h^2 \frac{1}{2} a_{k+1}^T M a_{k+1} + V(x_{k+1}) + \beta h^2 a_{k+1}^T \nabla V(x_{k+1})  + \frac{1}{2} \beta^2 h^4 a_{k+1}^T \text{Hess} V(x_{k+1}) a_{k+1} \big) \Big),
    \label{eq_var_VaLiPN}
\end{align}
then the Euler-Lagrange equation of the variational principle
\[
    \delta \sum_{k=1}^N \tilde{\mathcal{L}}_d(x_k, x_{k+1}, a_k, a_{k+1}) = 0
\]
yields the following symplectic method (originally stated in Section 3 in \cite{ZhSk97}):
\begin{Integrator} \textbf{Zhang and Skeel's symplectic method (Z\&S):}
    \begin{equation}
    \begin{cases}
        x_{k+1} &=x_k+h v_k+\frac{1}{2}h^2 f_k \\
        v_{k+1} &=v_k+\frac{1}{2}h(f_k+f_{k+1}) \\
        a_k &=-M^{-1} \nabla V(x_k)- M^{-1} \emph{Hess} V(x_k) \beta h^2 a_k \\
        f_k &= a_k-\frac{1}{2}\beta^2 h^4 M^{-1} a_k \cdot V^{(3)}(x_k) \cdot a_k
    \end{cases}
    \label{SyLiPNeq}
    \end{equation}
    \label{SyLiPN}
\end{Integrator}
\noindent where $V^{(3)}(\cdot)$ is a 3rd-order tensor corresponding to 3rd-derivative of $V$, the symbol $\cdot$ stands for tensor contraction, and therefore $a_k \cdot V^{(3)}(x_k) \cdot a_k$ is again a vector.

For computational efficiency, $a_k$ should be obtained by solving a symmetric linear system instead of inverting a matrix. In this sense, Z\&S is linearly implicit.

\begin{Theorem} Z\&S is:
    \begin{enumerate}
    \item unconditionally linearly stable if $\beta\geq 1/4$.
    \item variational (and thus symplectic and conserving momentum-maps).
    \item 2nd-order convergent (if stable) and can be made arbitrarily high order convergent.
    \item symmetric (``time-reversible'').
    \end{enumerate}
\end{Theorem}

Showing Z\&S is variational ensures\footnote{In general, symplectic methods are only locally variational, but variational methods are symplectic; see for instance \cite{MaWe:01}.}, due to a discrete Noether theorem (e.g., \cite{MaWe:01} and \cite{Hairer06}), that it also preserves momentum maps that correspond to system symmetries. This additional preservation property is desired in mechanical system simulations. The variational formulation also leads to a possible extension to Lie groups (Section \ref{sec_LieGroup}).

\paragraph{Constrained dynamics:}
One of our motivations for studying Z\&S originates from a need for coarse time-steppings in penalty methods for constrained dynamics.

To model constrained dynamics, let
\begin{equation}
    \mathcal{S}(q(t)):=\int_a^b \frac{1}{2} \dot{q}(t)^T M \dot{q}(t) - V(q(t)) \, dt
\end{equation}
be the action associated with system \eqref{NewtonEq}.
Under a holonomic constraint $g(q)=0$, the system evolution coincides with the critical trajectory on the constraint manifold, i.e., the solution of:
\begin{equation}
    \delta \mathcal{S} / \delta q = 0 \text{ and for all $t$, } q(t)\in g^{-1}(0) .
    \label{eqConstrainedDynamics}
\end{equation}
This trajectory can also be obtained by solving the differential algebraic system
\begin{equation}
\begin{cases}
  M\ddot{q}=-\nabla V(q)+\lambda^T \nabla g(q) \\
  g(q)=0
\end{cases}
\label{DAE}
\end{equation}

Penalty methods approximate rigid constraints by stiff potentials; this is a classical idea  and we refer to \cite{RuUn57,Takens80,TePl87,PlBa88} for a non-comprehensive list of references. More precisely, modify the potential energy $V(q)$ to $V(q)+\frac{1}{2}\omega^2 g(q)^T g(q)$, then the solution of \eqref{eqConstrainedDynamics} is approximated by the solution of the following unconstrained mechanical system:
\begin{equation}
  M\ddot{q}^{\omega}=-\nabla V(q^{\omega})-\omega^2 g(q^{\omega})^T \nabla g(q^{\omega}),
  \label{ourConstrainedDynamics}
\end{equation}
where  $\omega$ is large enough.
Paraphrasing \cite{PlBa88},  the problem is, ``\emph{as a result of stiffness, the numerical differential equation solver takes very small time steps, using a large amount of computing time without getting much done}''.
Z\&S alleviates this problem because it can use coarse time-steps and do not solve nonlinear systems (see Section \ref{SectionConstrained}).

On a related matter, although penalty methods are widely employed and proved convergent to constrained dynamics (see Section \ref{SectionPenalty}), a quantitative analysis of its link to Lagrange multiplier approach \eqref{DAE} appeared to be lacking.
We show that the solution of \eqref{ourConstrainedDynamics} converges to that of \eqref{DAE} as $\omega\rightarrow \infty$ in the sense of two-scale Flow convergence (see Definition 1.1 in \cite{FLAVOR10}). More precisely, we have (explained in Section \ref{ConstrainedDynamicsSection}; throughout this paper, `bounded' means having a norm bounded by an $\omega$-independent constant):

\begin{Theorem}
    Denote by $q^\omega(t)$ the solution to \eqref{ourConstrainedDynamics} with $q^\omega(0)=q_0$ and $\dot{q}^\omega(0)=\dot{q}_0$ (where $g(q_0)=0$ and $\frac{d}{dt}g(q_0)=\nabla g(q_0)\cdot \dot{q}_0=0$). Suppose that $M$ is non-singular, $V(\cdot)$ is bounded from below,  $V(q)$ diverges towards infinity as $|q|\rightarrow \infty$, $V(\cdot)$ and $g(\cdot)$ are $C^2$ with bounded derivatives, and for all $q\in g^{-1}(0)$, $\nabla g(q)$ has a constant rank equal to the codimension of the constraint manifold, then
    \begin{equation}
        \lambda(t):=-\lim_{T\rightarrow 0} \lim_{\omega\rightarrow \infty} \frac{1}{T}\int_t^{t+T} \omega^2 g(q^\omega(s))\, ds
        \label{asymptoticLambda}
    \end{equation}
    exists. Also, the solution $q(t)$ of
    \begin{equation}
    \begin{cases}
        M\ddot{q}(t)=-\nabla V(q(t))+\lambda(t)^T \nabla g(q(t)) \\
        g(q(t))=0,  \qquad  q(0)=q_0, \qquad \dot{q}(0)=\dot{q}_0
    \end{cases}
    \label{DAE2}
    \end{equation}
    exists and satisfies
    \begin{equation}
        q^\omega \xrightarrow{F} q
    \end{equation}
    in the sense of two-scale Flow convergence \cite{FLAVOR10}, i.e., for all bounded $t\geq 0$ and all bounded and uniformly Lipschitz continuous test function $\varphi$,
    \begin{equation}
        \lim_{T\rightarrow 0} \lim_{\omega\rightarrow \infty} \frac{1}{T}\int_t^{t+T} \varphi(q^\omega(s)) - \varphi(q(s)) \, ds = 0
        \label{eq_Fconvergence}
    \end{equation}
    \label{asymptoticLambdaTheorem}
\end{Theorem}

\paragraph{Outline of the paper:} Section \ref{SectionIntegrator} derives Z\&S from a variational principle, relates it to Newmark integrators, and discusses its properties. Section \ref{ConstrainedDynamicsSection} illustrates how penalty methods converge to Lagrange multiplier approach. Section \ref{exampleSection} applies the method to constrained systems (pendula and water molecular dynamics), a \text{non}-constrained model of DNA division, and a mechanical system on SO(3) illustrating the benefits of a variational formulation.

\subsection{On penalty methods}
\label{SectionPenalty}

The penalty strategy of replacing holonomic constraints by stiff potentials is widely used. For example, it is a common treatment in computer graphics (e.g., \cite{TePl87,WiFl87,PlBa88}).

It is known that the penalized solution converges to constrained dynamics in $C^1$ topology, as long as its initial condition is in tangent bundle of the constraint manifold. We refer to, e.g., the pioneering work of \cite{RuUn57, Takens80},  to \cite{BoSc97, Bo98, ShZe02} for recent progress, and to Chapter XIV.3 of \cite{Hairer06} for a review.

The reverse point of view has also been employed, particularly in molecular dynamics, where stiff oscillatory molecular bonds are replaced by rigid constraints for the purpose of allowing larger time-steps (e.g., \cite{Fixman:74,Sc10}). If the initial velocity is not in the tangent plane then a correction potential might also be required to account for the non-zero normal energy (e.g., \cite{Fixman:74, Re95, ScBo97homogenization}). The Fixman potential \cite{Fixman:74} is a classical example of such a correction, in particular when investigating thermodynamic properties of molecular systems (see e.g., \cite{BaFrLe11}); on the other hand,  \cite{BoSc95} suggests that Fixman might not be the right correction for deterministic systems.

\subsection{One constrained dynamics}
\label{SectionConstrained}

Other popular constrained dynamics
methods  include:  generalized coordinates on the constraint manifold (e.g., \cite{JaVaRo93}) and Lagrange multipliers (e.g., SHAKE \cite{RyCiBe97}, RATTLE \cite{An83}, SETTLE \cite{HeBeBeFr97}, LINCS \cite{MiKo92}, M-SHAKE \cite{KrGu01}). The equivalence between these two approaches is well-established (e.g., \cite{WeMa97}).
These numerical methods allow an $o(1)$ integration step, but they also require solving nonlinear systems at each step. Unfortunately, linearization of these methods are no longer symplectic, and therefore resorting to linearization for a speed-up is at the risk of losing long-time accuracy.

The advantage of using a penalty approach depends on the system: if the system has a large number of coupled constraints, then an integration of the penalized system, even with small steps, would still be faster than generalized coordinate and Lagrange multiplier methods, which require solving high dimensional nonlinear systems.

Z\&S provides a compromise by allowing large integration steps ($o(1)$, independent of $\omega$) with limited cost of a linear solve per iteration.
It remains  accurate when applied to penalized system \eqref{ourConstrainedDynamics}, even though the $o(1)$ step does not resolve stiffness of the equation. This is because stiffness in this system results in fast oscillations non-tangent to a stable slow manifold (Section \ref{ConstrainedDynamicsSection}). Although implicit methods damp high frequencies in oscillations (e.g., \cite{HaWa96}), the approximation of fast oscillations by slower ones (as in \cite{LiAbE:08})
is sufficient  for the approximation of slow dynamics on the constraint manifold.

 We refer to M-SHAKE \cite{KrGu01} for an example of recent developments to the Lagrange multiplier method. While M-SHAKE  is limited to systems with distance constraints, Z\&S combined with penalty method can implement arbitrary holonomic constraints.

\section{Z\&S: structure-preserving and stable integrators}
\label{SectionIntegrator}
\subsection{Derivation from Newmark integrators}
\label{Sectionderivation}

The Newmark family of algorithms are extensively used in structural dynamics \cite{Ne59}:
\begin{Integrator} \textbf{Newmark:}
    \begin{equation}
    \begin{cases}
        q_{k+1} &=q_k+h\dot{q}_k+\frac{h^2}{2}[(1-2\beta)a_k+2\beta a_{k+1}] \\
        \dot{q}_{k+1} &=\dot{q}_k+h[(1-\gamma)a_k+\gamma a_{k+1}] \\
        a_k &=-M^{-1}\nabla V(q_k)
    \end{cases}
    \end{equation}
    \label{Newmark}
\end{Integrator}
Newmark is generally implicit when $\beta\neq 0$. When $\gamma=1/2$, it is 2nd-order accurate and variational \cite{KaMaOrWe00}, and we restrict ourselves to this case in this paper. Integrator \ref{Newmark} does not preserve the canonical symplectic form, and it was shown \cite{SkZhSc97,MaWe:01} that if one pushes forward the update map by a coordinate transform $\eta:TQ \rightarrow TQ$ defined as
\begin{equation}
    (x,v):=\eta (q,\dot{q})=(q+\beta h^2 M^{-1}\nabla V(q),\dot{q}) \quad,
\end{equation}
then we obtain an integrator that preserves the canonical symplectic form on $T^*Q$:
\begin{Integrator} \textbf{Push-forward Newmark:}
    \begin{equation}
    \begin{cases}
        x_{k+1} &=x_k+h v_k+\frac{1}{2}h^2 a_k \\
        v_{k+1} &=v_k+\frac{1}{2}h(a_k+a_{k+1}) \\
        a_k &=-M^{-1} \nabla V(x_k+\beta h^2 a_k)
    \end{cases}
    \label{eq_PN}
    \end{equation}
\end{Integrator}
These two methods are unconditionally linearly stable if $\beta\geq 1/4$ \cite{ChKa02,SkZhSc97}. Newmark with $\beta\geq 1/4$ is known to be nonlinearly stable under specific conditions \cite{Hu77}, and Push-forward Newmark is known to be stable near stable fixed points in non-resonant nonlinear settings \cite{SkSr00}. Nevertheless, there are nonlinear cases in which Newmark is no longer stable \cite{ErBoBu02,KuRa96}. In fact, few convergent methods are unconditionally stable for arbitrary nonlinear systems to the authors' knowledge (see also \cite{WoOd88}).

Now, consider the following discrete Lagrangian (see \cite{MaWe:01} for a review of variational integrators and \cite{AbMa08} for one of many excellent reviews of analytical mechanics):
\begin{align}
    & \mathcal{L}_d(x_k, x_{k+1}, a_k, a_{k+1})=h \Big( \frac{1}{2} \big( \frac{x_{k+1}-x_k}{h} \big)^T M \big( \frac{x_{k+1}-x_k}{h} \big) - \frac{1}{2} \big( \beta h^2 \frac{1}{2} a_k^T M a_k + V(x_k+\beta h^2 a_k) \big) \nonumber\\
    & \qquad \qquad - \frac{1}{2} \big( \beta h^2 \frac{1}{2} a_{k+1}^T M a_{k+1} + V(x_{k+1}+\beta h^2 a_{k+1}) \big) \Big),
    \label{eq_var_PN}
\end{align}
then, \eqref{eq_PN} is the associated Euler-Lagrange equation, i.e., the critical point of discretized action $\sum_k \mathcal{L}_d (x_k, x_{k+1}, a_k, a_{k+1})$. Note this variational formulation is explicit and  distinct from the one implicitly defined in \cite{MaWe:01}.

However, \eqref{eq_var_PN} is still implicit. Therefore we use  a 2nd order Taylor expansion of \eqref{eq_var_PN} and derive \eqref{eq_var_VaLiPN}. To obtain the corresponding discrete Euler-Lagrange equation, we compute $\partial \tilde{\mathcal{L}}_d / \partial a_k = 0$, which leads to
\begin{equation}
    M a_k + \nabla V(x_k) + \beta h^2 \text{Hess} V(x_k) a_k = 0 .
    \label{eq_VaLiPN_a}
\end{equation}
We then compute  the discrete Legendre transform (see \cite{MaWe:01} for notation and terminology), which introduces the momentum and leads to
\begin{equation}
 \begin{cases}
    p_k &=-D_1 L_d(x_k,x_{k+1},a_k,a_{k+1}) \\
        &= M \frac{x_{k+1}-x_k}{h}+\frac{h}{2}\left(-M a_k + \frac{1}{2}\beta^2 h^4 a_k \cdot V^{(3)}(x_k) \cdot a_k \right) \\
    p_{k+1} &= D_2 L_d(x_k,x_{k+1},a_k,a_{k+1}) \\
            &= M \frac{x_{k+1}-x_k}{h}-\frac{h}{2}\left(-M a_{k+1} + \frac{1}{2}\beta^2 h^4 a_{k+1} \cdot V^{(3)}(x_{k+1}) \cdot a_{k+1} \right)
 \end{cases}
 \label{eq_VaLiPN_p}
\end{equation}
Since the velocity and momentum are related via $v_k=M^{-1} p_k$, we obtain the Z\&S update \eqref{SyLiPNeq}.

Because the new discrete Lagrangian $\tilde{\mathcal{L}}_d$ is quadratic in $a$, nonlinear solves in Push-forward Newmark are replaced by linear solves in Z\&S. Consequently, Z\&S exhibits a speed advantage. Numerical illustrations of this advantage are in Section \ref{exampleSection}.

\subsection{On Z\&S}
\label{SectionViewAngles}
\paragraph{Linearizing equations:} Z\&S is obtained as the linearization of Push-forward Newmark update map combined with a small correction. More precisely, the Taylor expansion of Line 3 of \eqref{eq_PN} leads to \eqref{eq_VaLiPN_a}, and Lines 1-2 in \eqref{eq_PN} can be rewritten in terms of momentum as
\[
 \begin{cases}
    p_k &=-D_1 L_d(x_k,x_{k+1},a_k,a_{k+1}) \\
        &= M \frac{x_{k+1}-x_k}{h}+\frac{h}{2}\left(-M a_k \right) \\
    p_{k+1} &= D_2 L_d(x_k,x_{k+1},a_k,a_{k+1}) \\
            &= M \frac{x_{k+1}-x_k}{h}-\frac{h}{2}\left(-M a_{k+1} \right)
 \end{cases}
\]
The difference are two $\mathcal{O}(h^5)$ terms (corresponding to $\frac{1}{4}\beta^2 h^5 a \cdot V^{(3)}(x) \cdot a$). To be consistent with the literature, we summarize this variant using velocity instead of momentum:

\begin{Integrator} \textbf{Zhang and Skeel's method simplified (Z\&Ss):}
    \begin{equation}
    \begin{cases}
        x_{k+1} &=x_k+h v_k+\frac{1}{2}h^2 a_k \\
        v_{k+1} &=v_k+\frac{1}{2}h(a_k+a_{k+1}) \\
        a_k &=-(M+\emph{Hess} V(x_k) \beta h^2)^{-1} \nabla V(x_k)
    \end{cases}
    \label{oldSyLiPNeq}
    \end{equation}
    \label{SyLiPN2}
\end{Integrator}
\begin{Theorem} Z\&Ss is:
    \begin{enumerate}
    \item unconditionally linearly stable if $\beta\geq 1/4$.
    \item symplectic, if the $n\times n$ matrix $M+\emph{Hess}V(x)\beta h^2$ commutes with the $n\times n$ matrix $V^{(3)}(x) \cdot (M+\emph{Hess}V(x)\beta h^2)^{-1} \cdot \nabla V(x)$ ($\cdot$ is tensor contraction).
    \item 2nd-order convergent (if stable) and can be made arbitrarily high order convergent.
    \item symmetric (``time-reversible'').
    \end{enumerate}
    \label{thm_SyLiPN2}
\end{Theorem}
\noindent
Z\&Ss is not always symplectic due to the removal of $\mathcal{O}(h^5)$ terms. However, it requires no high-order tensor operations and is thus a good choice for high-dimensional problems. 

\paragraph{Partial Newton solve:} Line 3 of Z\&Ss can be viewed as executing only the first step of a Newton solver for the nonlinear equation $a_k =-M^{-1} \nabla V(x_k+\beta h^2 a_k)$.

\paragraph{Preconditioning, filtering and regularization:} The factor of $(M+\text{Hess}V(x)\beta h^2)^{-1}$ in front of $\nabla V(x)$ can be thought as an optimization preconditioner or a way to filter \cite{HeGoGo07,FaGr11} / regularize \cite{Skeel:99, Sanz-Serna:08} high frequency oscillations.

\subsection{Properties}
(Proofs of results introduced in this paragraph are standard and available online at \url{http://www.math.gatech.edu/~mtao/TaOw14_supplemental.pdf}).
\begin{Theorem}[Stability]
    Z\&S (Integrator \ref{SyLiPN}) is unconditionally linearly stable if and only if $\beta\geq 1/4$.
\end{Theorem}
The proof of the unconditional linearly stability (for $\beta\geq 1/4$)  of Integrators \ref{SyLiPN2} and  \ref{SyLiPN} are similar.
 If the potential is of form $V(x)=V_0(x)+\epsilon^{-1} V_1(x)$, then the following modification of Z\&Ss is  unconditionally linearly stable\footnote{in the sense that the solution remains bounded for all $h$ when $V_0$ has Lipschitz-continuous 1st-derivative with bounded Lipschitz constant and $V_1$ is quadratic and positive definite.} as long as $\beta>1/4+\mathcal{O}(\epsilon)$:
    \begin{Integrator} \textbf{Simplified Z\&Ss for stiff systems ($\epsilon^{-1}\gg 1$):}
        \begin{equation}
        \begin{cases}
            x_{k+1} &=x_k+h v_k+\frac{1}{2}h^2 a_k \\
            v_{k+1} &=v_k+\frac{1}{2}h(a_k+a_{k+1}) \\
            a_k &=-M^{-1} (\nabla V_0(x_k)+\epsilon^{-1}\nabla V_1(x_k) )- M^{-1} \epsilon^{-1}\emph{Hess} V_1(x_k) \beta h^2 a_k
        \end{cases}
        \end{equation}
        \label{SyLiPN3}
    \end{Integrator}

\begin{Theorem}[Consistency]
    Consider an integrator for \eqref{NewtonEq} given by:
    \begin{equation}
    \begin{cases}
        x_{k+1} &=x_k+h v_k+\frac{1}{2}h^2 a_k \\
        v_{k+1} &=v_k+\frac{1}{2}h(a_k+a_{k+1}+h^4 g(x_k)+h^4 g(x_{k+1})) \\
        a_k &=-M^{-1} \nabla V(x_k)- M^{-1} f(x_k) h^2 a_k
    \end{cases},
    \label{f9g3uinlvu3bofbu1og}
    \end{equation}
    where $f, g\in \mathcal{C}(Q)$ are arbitrary functions. If $V\in \mathcal{C}^3(Q)$, this integrator has 3rd order truncation error.
    \label{GeneralAccuracy}
\end{Theorem}

\begin{Corollary}
    Z\&S (Integrator \ref{SyLiPN}), Z\&Ss (Integrator \ref{SyLiPN2}) and simplified Z\&Ss for stiff systems (Integrator \ref{SyLiPN3}) are 2nd order convergent, provided stability.
    \label{TheoremLocalError}
\end{Corollary}
Symmetry (i.e., time-reversibility) is one desired property of numerical integrators, because it leads to good long time performance (see for instance \cite{Hairer06} or \cite{MR2132573}).

\begin{Theorem}[Symmetry / Time-Reversibility]
    Let $f\in \mathcal{C}^1(Q)$ is an arbitrary function. The integrator defined by
    \begin{equation}
    \begin{cases}
        x_{k+1} &=x_k+h v_k+\frac{1}{2}h^2 f_k \\
        v_{k+1} &=v_k+\frac{1}{2}h(f_k+f_{k+1}) \\
        f_k &=f(x_k)
    \end{cases},
    \label{n413tad15sigpquhrgp1}
    \end{equation}
     is symmetric (time-reversible).
    \label{GeneralSymmetricIntegrators}
\end{Theorem}

\begin{Corollary}
    Z\&S (Integrator \ref{SyLiPN}) is symmetric (time-reversible).
\end{Corollary}

\begin{Remark}
    Arbitrary high order Z\&S can be obtained using standard splitting schemes as in \cite{Yoshida:90,Neri:88,Hairer06}. A 4th-order example is provided in the supplemental material.
    \label{RemarkHighOrder}
\end{Remark}

\begin{Theorem}
    Z\&S (Integrator \ref{SyLiPN}) is symplectic.
    \label{TheoremSymplecticity}
\end{Theorem}

\begin{Lemma}[Symplecticity]
    Consider an integrator given by \eqref{n413tad15sigpquhrgp1}.
    If $f\in \mathcal{C}^1(Q)$ is a function with symmetric Jacobian, then this integrator is symplectic.
    \label{GeneralSymplecticIntegrators}
\end{Lemma}

\begin{Remark}
    The commutation condition in Theorem \ref{thm_SyLiPN2} ensures a symmetric Jacobian and hence the symplecticity of Z\&Ss. Two very special cases where this condition is satisfied are: when the system contains only 1 degree of freedom, or when $\text{Hess}(V)$ can be diagonalized by a matrix independent of $x$.
\end{Remark}

\begin{Remark}
    Fully-nonlinear implicit symplectic methods (e.g., midpoint or Newmark) are not exactly symplectic due to numerical errors in nonlinear solves, which are oftentimes much larger than those in linear solves.
\end{Remark}

\section{Lagrange multiplier methods as limits of penalty methods}
\label{ConstrainedDynamicsSection}
Lagrange multiplier and penalty methods respectively simulate \eqref{DAE} and \eqref{ourConstrainedDynamics}. It is known (Section \ref{SectionPenalty} and \ref{SectionConstrained}) that both are equivalent to constrained dynamics (in the $\omega\rightarrow\infty$ limit). We now quantify the equivalence between themselves.

First observe that this equivalence is not necessarily achieved via
\begin{equation}
    \lambda(t)=-\lim_{\omega\rightarrow\infty} \omega^2 g(q^\omega(t))
    \label{v9a87dfgfoiyu3bogtuib1o}
\end{equation}
Consider for instance $V=0$, $g(q)=q$, $q^\omega(0)=1/\omega^2$ and $p^\omega(0)=0$; \eqref{ourConstrainedDynamics} leads to $q^\omega(t)=\cos(\omega t)/\omega^2$, and \eqref{DAE} yields $q(t)=0$ and $\lambda(t)=0$, but then \eqref{v9a87dfgfoiyu3bogtuib1o} cannot hold because $\lim_{\omega\rightarrow\infty} \omega^2 g(q^\omega(t))$ does not exist.

The appropriate notion of equivalence is provided Theorem \ref{asymptoticLambdaTheorem}. The idea is as follows: energy conservation implies that $g(q^\omega)$ is at most $\mathcal{O}(1/\omega)$ (see Lemma \ref{Lemma1}). In fact, $g(q^\omega)$ can be further shown to be $\mathcal{O}(1/\omega^2)$ (see Remark \ref{p8ohrfqilbvlorgbgouob14} or \cite{KeCo96}), and constraints are satisfied with small errors that oscillate rapidly. To describe the Lagrange multiplier system \eqref{DAE} as a limit of penalized systems \eqref{ourConstrainedDynamics}, the convergence of these fast oscillations should be understood in a weak sense, whereas slow dynamics on the constrained manifold converges strongly. Thus we employ two-scale Flow convergence (Eq.\ref{eq_Fconvergence}) for this description. Convergence is first proved for flat constraint manifold (Lemma \ref{Lemma2}), and then local charts are patched together (see Appendix); this leads to Theorem \ref{asymptoticLambdaTheorem}.

\begin{Remark}
    If the limit in \eqref{v9a87dfgfoiyu3bogtuib1o} exists, then \eqref{asymptoticLambda} simplifies to \eqref{v9a87dfgfoiyu3bogtuib1o} and two-scale F-convergence becomes strong convergence. However, we are not aware of a practical example where such a limit exists.
\end{Remark}

\section{Application examples}
\label{exampleSection}
Z\&S (Integrator \ref{SyLiPN}) is applied in Sections \ref{SectionDoublePendulum},\ref{SectionMultiplependula} and \ref{SectionDNAexample}; Section \ref{WaterExample} employs simplified Z\&Ss for stiff system (Integrator \ref{SyLiPN3}) due to its efficiency for high-dimensional systems; Section \ref{sec_LieGroup} is based on variational formulation \eqref{eq_var_VaLiPN}.

Speed comparisons are provided in terms of running times (using Matlab 7.7 on an Intel Core 2 Duo 2.4G laptop, with nonlinear solver of `fsolve'); however, these numbers are machine and platform dependent (e.g., Matlab is very well-optimized for linear algebra), and should serve only as a qualitative illustration of efficiencies.



\subsection{Double pendulum}
\label{SectionDoublePendulum}

\paragraph{Implementation:}
One way to represent planar double pendulum is to use 4 degrees of freedom and 2 nonlinear constraints. Using the notations of \eqref{ourConstrainedDynamics}, we have
\[
    M=\begin{bmatrix} m_1 & 0 & 0 & 0 \\ 0 & m_1 & 0 & 0 \\ 0 & 0 & m_2 & 0 \\ 0 & 0 & 0 & m_2 \end{bmatrix},   \qquad
    \begin{aligned}
        &V(x_1,y_1,x_2,y_2) = -\text{g}y_1-\text{g}y_2, \\
        &g(x_1,y_1,x_2,y_2)=\begin{bmatrix} x_1^2+y_1^2-L_1^2 \\ (x_2-x_1)^2+(y_2-y_1)^2-L_2^2 \end{bmatrix} .
    \end{aligned}
\]
For simplicity, we adopt a dimensionless convention and assume $m_1=m_2=\text{g}=1$.

The Z\&S simulation of the penalized system \eqref{ourConstrainedDynamics} is straightforward. SHAKE \cite{RyCiBe97} is used as the Lagrangian multiplier method in our experiments; it is nonlinearly implicit.

Symplectic integration in generalized coordinates $\theta, \phi$ ($x_1=L_1\sin\theta$, $y_1=-L_1\cos\theta$, $x_2=L_1\sin\theta+L_2\sin\phi$, $y_2=-L_1\cos\theta-L_2\cos\phi$) is also implicit. This is because, after writing down the Lagrangian, one will note a position-dependent mass matrix of
\begin{equation}
    \tilde{M}(\theta,\phi)=\begin{bmatrix} 2L_1^2 & L_1L_2(\cos\theta\cos\phi+\sin\theta\sin\phi) \\ L_1L_2(\cos\theta\cos\phi+\sin\theta\sin\phi) & L_2^2 \end{bmatrix},
    \label{ExpressionMassMatrix}
\end{equation}
Consequently, even the most well known ``explicit'' variational integrators such as variational Euler (i.e., leapfrog) and Velocity-Verlet, will be implicit.

Note although $g$ is quadratic, the penalized ODE is cubically nonlinear.

\begin{figure}[h]
\includegraphics{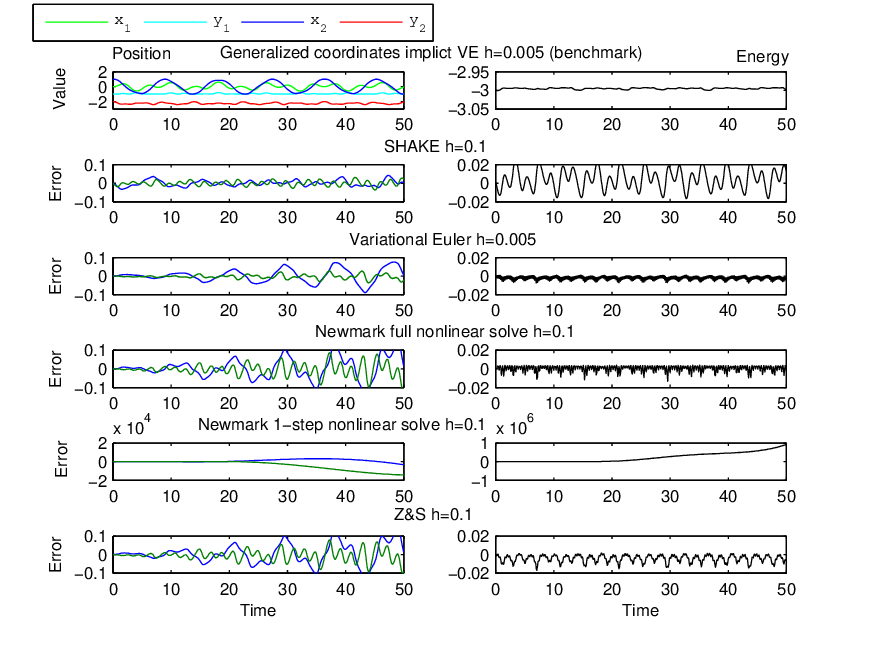}
\caption{\footnotesize Errors of SHAKE, Variational Euler on the penalized system, Newmark (with nonlinear systems fully solved), linearly-implicit Newmark (with only 1st iteration of nonlinear solve at each step), and Z\&S. Benchmark is provided by small step Variational Euler in generalized coordinates. Initial conditions are $x_1(0)=0,y_1(0)=-1,x_2(0)=1,y_2(0)=-2$, zero momenta; $L_1=1$ and $L_2=\sqrt{2}$. $\omega=20$ in Rows 3-6. $\beta=0.4$. Row 3 uses $h=0.1/\omega$ for stability, and Row 2,4,5,6 use a 20x bigger $h=0.1$. Position errors are only shown on $x_2$ and $y_2$ for readability.}
\label{comparison}
\end{figure}

%
%

\begin{figure}
\centering
\footnotesize
\subfigure[$\omega^2\|g(q(T))\|$ numerically computed as a function of $\omega$. $T=50$ fixed.]{
\includegraphics[width=0.45\textwidth]{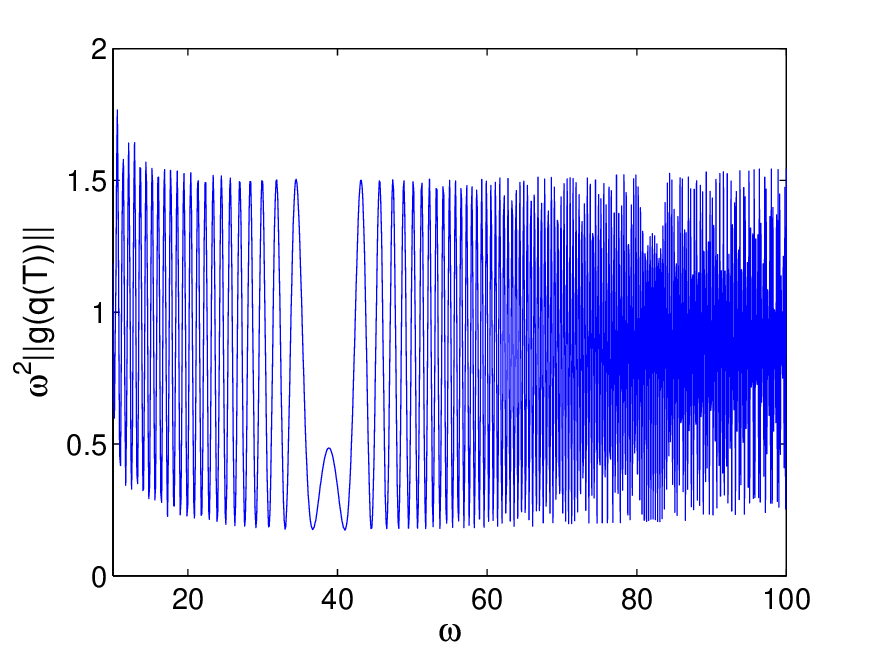}
\label{DoublePendulumConstraintOmegaSqr}
}
\quad
\subfigure[Lagrange multipliers computed by \eqref{asymptoticLambda} from $\omega=20$ penalized system and by SHAKE. The integral in \eqref{asymptoticLambda} is approximated by empirical average over time window of width 0.2.]{
\includegraphics[width=0.45\textwidth]{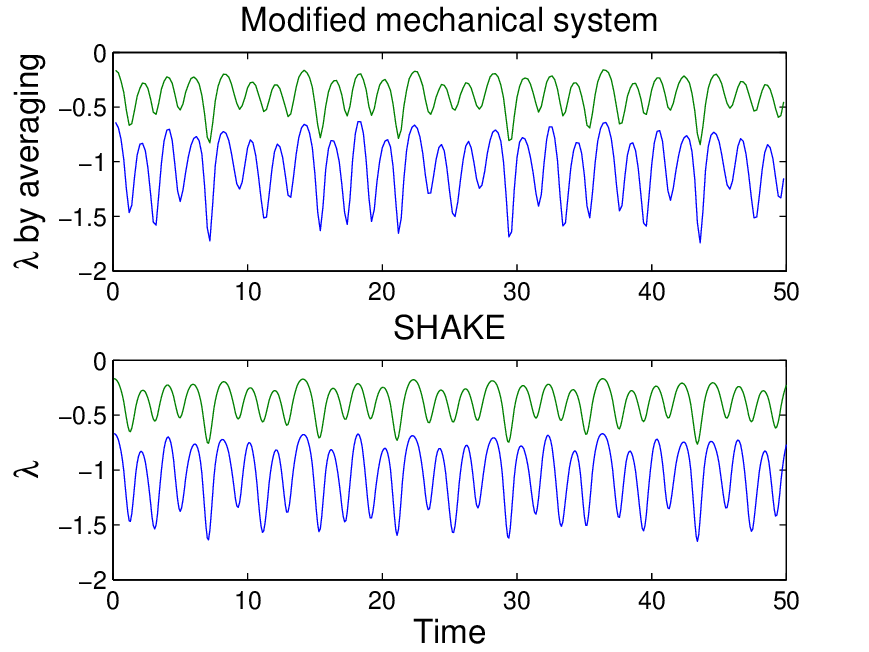}
\label{AsymptoticLambda}
}
\vspace{-8pt}
\caption{\footnotesize Satisfaction of constraints and Lagrangian multiplier. Z\&S with $h=0.01$ is used (for smooth curve; Verlet would require $h=0.001$ for stability); other parameters are same as in Figure \ref{comparison}.}
\end{figure}

\paragraph{Results:}
Figure \ref{comparison} illustrates errors of different methods. Newmark (Integrator \ref{Newmark}) with only 1st step of nonlinear solve (Row 5) has a large error due to the loss of symplecticity, even though the method is still consistent. On the contrary, Z\&S (Row 6) yields small errors almost identical to those of fully-nonlinear-solve Newmark (Row 4).

Z\&S produces larger error than SHAKE, because there is modeling error due to finite $\omega$ in addition to integration error (Row 3). We chose an intermediate $\omega$, which is sufficiently large to approximate the constraints, yet small enough to show that the penalized system is only an approximation. A larger $\omega$ leads to a more accurate approximation, but if it is too large, e.g., $\omega=2000$ (i.e., a stiffness of $\omega^2=4\times 10^6$), instability occurs in all Z\&S, original Newmark, and implicit midpoint due to strong nonlinearity.

If $\omega$ is finite, the approximation error is predicted to be $\mathcal{O}(\omega^{-2})$ (Remark \ref{p8ohrfqilbvlorgbgouob14}). See Figure \ref{DoublePendulumConstraintOmegaSqr} for a numerical illustration. Figure \ref{AsymptoticLambda} compares the Lagrange multiplier computed by SHAKE with the one obtained from the penalized system via Theorem \ref{asymptoticLambdaTheorem}. There is no strong convergence but only a 2-scale F-convergence.


It is known that the double pendulum contains a chaotic region (e.g., \cite{RiSc84}). Variational integrator is desired for simulating such systems \cite{ChSc90,McAt92}. None of our symplectic simulations (Row 1,2,3,4,6) led to numerical leakage between regular and chaotic regions.

Generalized coordinate implicit VE (benchmark), SHAKE, Variational Euler, Newmark with full nonlinear solve, Newmark with one-step nonlinear solve, and Z\&S respectively spent 91.3, 3.8, 1.0, 5.3, 0.2 (2.6 if $\text{Hess}V$ is not analytically provided but approximated by the nonlinear solver), and 0.4 seconds on the above simulation.


\subsection{A simple high-dimensional example: a chain of many pendula}
\label{SectionMultiplependula}
Consider a chain of $n$ pendula, which approximates a continuous rope. The system is similarly modeled by \eqref{ourConstrainedDynamics} with
\[
    M=\begin{bmatrix} 1 & 0 & \cdots & 0 & 0 \\ 0 & 1 & \cdots & 0 & 0 \\ \vdots & \vdots & \ddots & \vdots & \vdots \\ 0 & 0 & \cdots & 1 & 0 \\ 0 & 0 & \cdots & 0 & 1 \end{bmatrix},   \quad
    \begin{aligned}
        & V(x_1,y_1,\cdots,x_n,y_n)=-\sum_{i=1}^n y_i, \\
        & g(x_1,y_1,\cdots,x_n,y_n)=\begin{bmatrix} x_1^2+y_1^2-L_1^2 \\ (x_2-x_1)^2+(y_2-y_1)^2-L_2^2 \\ \vdots \\ (x_n-x_{n-1})^2+(y_n-y_{n-1})^2-L_n^2 \end{bmatrix} .
    \end{aligned}
\]

\begin{figure}[h]
\center
\includegraphics[width=0.8\textwidth]{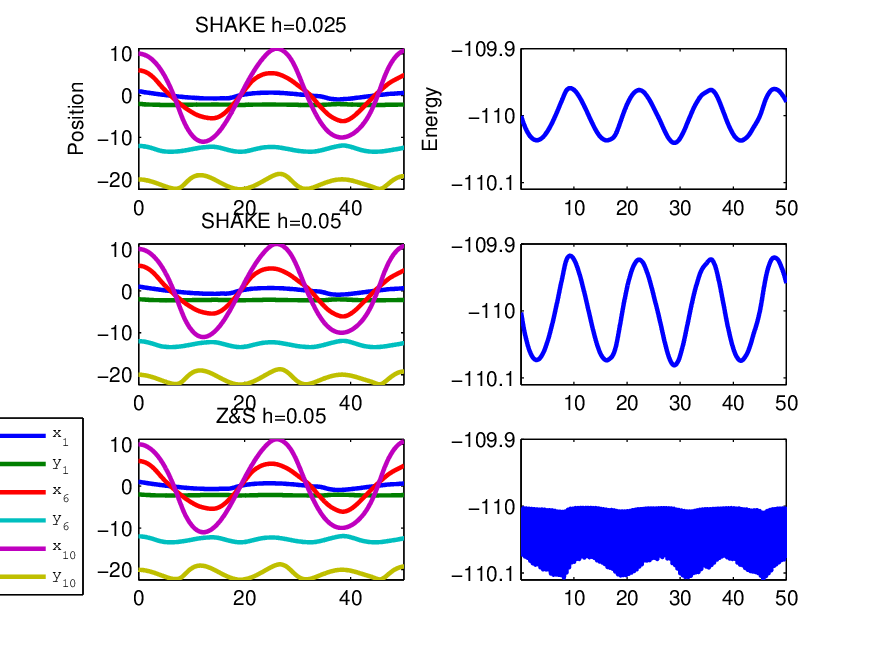}
\vspace{-12pt}
\caption{\footnotesize Simulations by SHAKE with $h=0.05$ and $h=0.025$ and by Z\&S with $h=0.05$ on $\omega=20$ penalized system. $n=10$; $x_i(0)=i,y_i(0)=-2i$ for $i=1,\ldots,n$ and initial momenta are zero; $L_i=\sqrt{5}$; $\beta=0.4$. For clarity, not all degrees of freedom are shown.}
\label{comparison2}
\end{figure}

Figure \ref{comparison2} shows good agreement between SHAKE and Z\&S (trajectories instead of errors are shown due to lack of accurate benchmark --- an analytical solution is unavailable, and Lagrange multiplier formulations and generalized coordinate approaches involve solving large nonlinear systems). SHAKE with $h=0.025$, $h=0.05$, and Z\&S with $h=0.05$ respectively spent 16.5, 8.1, and 1.1 seconds in these simulations.

\subsection{Molecular dynamics of water cluster}
\label{WaterExample}
Consider the dynamics of water molecules, each interacting with others via non-bonded interactions of electrostatic and van der Waals forces (both highly nonlinear).

\paragraph{Water model:} Use the popular TIP3P model (e.g., \cite{JoCh83})). Let $q_{ai}$ and $p_{ai}$ be position and momentum of $a$th molecule's $i$th atom (both 3-vectors). The Hamiltonian is:
\begin{equation}
    \mathcal{H}=\sum_{a=1}^N \sum_{i=1}^3 \frac{1}{2}p_{ai}^T m_i^{-1} p_{ai}+\sum_{a=1}^{N-1}\sum_{b=a+1}^{N} \left(
    \sum_{i=1}^3 \sum_{j=1}^3 \frac{K_c Q_i Q_j}{r_{ai,bj}}+\frac{A}{r_{a2,b2}^{12}}-\frac{C}{r_{a2,b2}^6}
    \right) ,
\end{equation}
where $r_{ai,bj}:=\|q_{ai}-q_{bj}\|$ is inter-atom distance, $m_1=m_3$ are hydrogen mass and $m_2$ oxygen, $K_c$ is electrostatic constant, $Q_i$ is partial charge of atom $i$ relative to electron charge, $A$ and $C$ are Lennard-Jones constants that approximate van der Waals forces.

In this TIP3P model (or many other prevailing models such as SPC, BF, TIPS2, and TIP4P, which are discussed in, e.g., \cite{JoCh83,SpMa98}), each water molecule is considered as a rigid body, with two O-H bond lengths and H-O-H bond angle fixed as constants $r_{OH}$ and $\alpha_{HOH}$. Detailed values of these model parameters could be found in, e.g., \cite{JoCh83}.

Therefore, the following vectorial constraint enforces the geometry of molecules:
\begin{equation}
    g(q)=\begin{bmatrix} (q_{11}-q_{12})(q_{11}-q_{12})^T-r_{OH}^2 \\
     (q_{13}-q_{12})(q_{13}-q_{12})^T-r_{OH}^2 \\ (q_{11}-q_{13})(q_{11}-q_{13})^T-r_{HH}^2 \\
     \vdots \\
     (q_{N1}-q_{N2})(q_{N1}-q_{N2})^T-r_{OH}^2 \\
     (q_{N3}-q_{N2})(q_{N3}-q_{N2})^T-r_{OH}^2 \\
     (q_{N1}-q_{N3})(q_{N1}-q_{N3})^T-r_{HH}^2 \end{bmatrix},
\end{equation}
where $r_{HH}:=2r_{OH} \sin(\alpha_{HOH}/2)$ is a constant. This leads to penalized Hamiltonian
\begin{equation}
    \tilde{\mathcal{H}}=\mathcal{H}+\frac{1}{2}\omega^2 \sum_{a=1}^N \left( (r_{a1,a2}^2-r_{OH}^2)^2+(r_{a3,a2}^2-r_{OH}^2)^2+(r_{a1,a3}^2-r_{HH}^2)^2 \right) .
    \label{pvqreuphp3410hrp1ubo}
\end{equation}

\paragraph{Constant temperature simulation:}
Constant temperature simulations are of practical importance because (i) thermal fluctuations are an indispensable component of molecular dynamics, and (ii) the N-body system is chaotic and its long-time deterministic simulation has limited predictive power. We use Langevin dynamics (e.g., \cite{Sc10}) as our constant temperature model.

In this model, molecules experience perturbation by noise and dissipation due to friction, and the dynamics can be expressed by the following SDEs:
\begin{equation}
\begin{cases}
    dq=\frac{\partial \tilde{\mathcal{H}}}{\partial p} \, dt \\
    dp=-\frac{\partial \tilde{\mathcal{H}}}{\partial q} \, dt -\gamma \frac{\partial \tilde{\mathcal{H}}}{\partial p} \, dt +\sqrt{2\gamma \beta^{-1}} dW
    \label{agoqo9834hgvqurbnp}
\end{cases},
\end{equation}
where $W$ is a $9N$-dimensional Wiener process, $\beta^{-1}>0$ is the constant temperature, and $\gamma>0$ is dissipation strength. The system admits an invariant measure of Boltzmann-Gibbs (BG; also known as the canonical ensemble), given by:
\begin{equation}
    \pi(q,p)=Z^{-1} \exp(-\beta \tilde{\mathcal{H}}),
    \label{BGdistribution}
\end{equation}
where $Z=\int_{\mathbb{R}^{18N}} \exp(-\beta \tilde{\mathcal{H}}) \, dq \, dp$ is the partition function.

To simulate \eqref{agoqo9834hgvqurbnp}, we use the  Geometric Langevin Algorithm (GLA; see \cite{BoOw:09}). GLA allows for an extension of Hamiltonian integrators to Langevin integrators. It is a splitting scheme, based on composing the one-step update of a determinstic integrator with the exact flow of an Ornstein-Uhlenbeck process (OU for short, given by $dp=-\gamma M^{-1} p \, dt +\sqrt{2\gamma \beta^{-1}} dW$, i.e., driftless noise and friction). It has been shown \cite{BoOw:09} that if the deterministic integrator is symplectic, then GLA  not only provides a good approximation of trajectories but also of BG (the invariant distribution). In this example, the deterministic building block is simplified Z\&Ss (Integrator \ref{SyLiPN3}) or SHAKE.

Constant temperature molecular dynamics is a rich research field, and our investigation will  only be numerical. The thermodynamic properties of a system with strong restraint may not be equivalent to those of a constrained system (e.g., \cite{BaFrLe11}); the Fixman potential is a classical way to correct the difference (see \cite{PeSkYa85} for debate on the validity of this correction). Proving that a numerical method samples a good approximation of the invariant distribution is nontrivial. \cite{BoOw:09} combines ergodicity  with the backward error analysis of symplectic integrators to show that the invariant distribution is preserved with a high order of accuracy. It is conjectured (Remark 2.1 in \cite{BoOw:09}) that SHAKE+GLA approximately samples a constrained BG distribution
\begin{equation}
    \hat{\pi}(q,p)=\hat{Z}^{-1} \exp(-\beta \mathcal{H}),
    \label{BGdistribution2}
\end{equation}
where $\hat{Z}=\int_{T^* g^{-1}(0)} \exp(-\beta \mathcal{H}) \, dq \, dp$.
Relating \eqref{BGdistribution} and \eqref{BGdistribution2} is left as a future investigation.
See \cite{VaCi2006, BaFrLe11, Ha08, LeRoSt12} for more about finite temperature constrained dynamics.


\paragraph{Numerical results:}
\begin{figure}[htb]
\includegraphics[width=\textwidth]{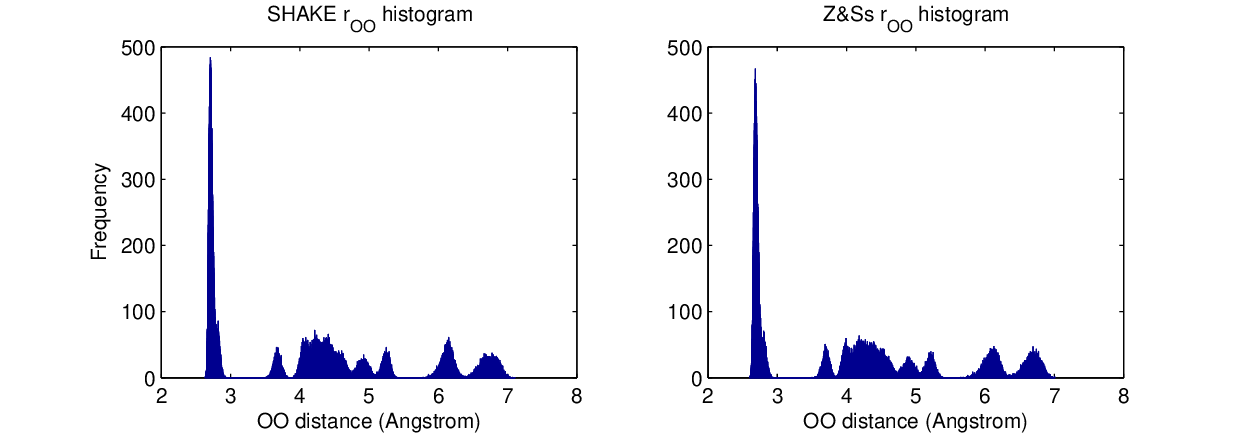}
\caption{\footnotesize Empirical OO radial distribution in 7-water cluster obtained by long time ($T=10000$) simulations of SHAKE and simplified Z\&Ss.}
\label{Water7OO}
\end{figure}

One quantity of interest in water cluster is the distribution of interatomic oxygen-oxygen distances in the thermal equilibrium limit, also known as the OO radial distribution \cite{JoCh83}. To illustrate the accuracy of Z\&S in sampling BG, Figure \ref{Water7OO} shows histograms obtained by long time simulations of SHAKE and simplified Z\&Ss (Integrator \eqref{SyLiPN3})
that approximate this distribution. We chose a system of size $N=7$  (i.e., $63$ degrees of freedom) so that peaks in the distribution could be clearly distinguished. SHAKE required 13472 secs, including 12284 secs on nonlinear solves (with tolerances of $10^{-6}$ on variable and $10^{-10}$ on function value), whereas simplified Z\&Ss used 1549 secs, including 67 secs on linear solves. Parameters are $\omega=20$, $h=0.05$ in both simulations, $\gamma=0.01$ and $\beta=50$.

\begin{figure}[htb]
\center
\includegraphics[width=0.75\textwidth]{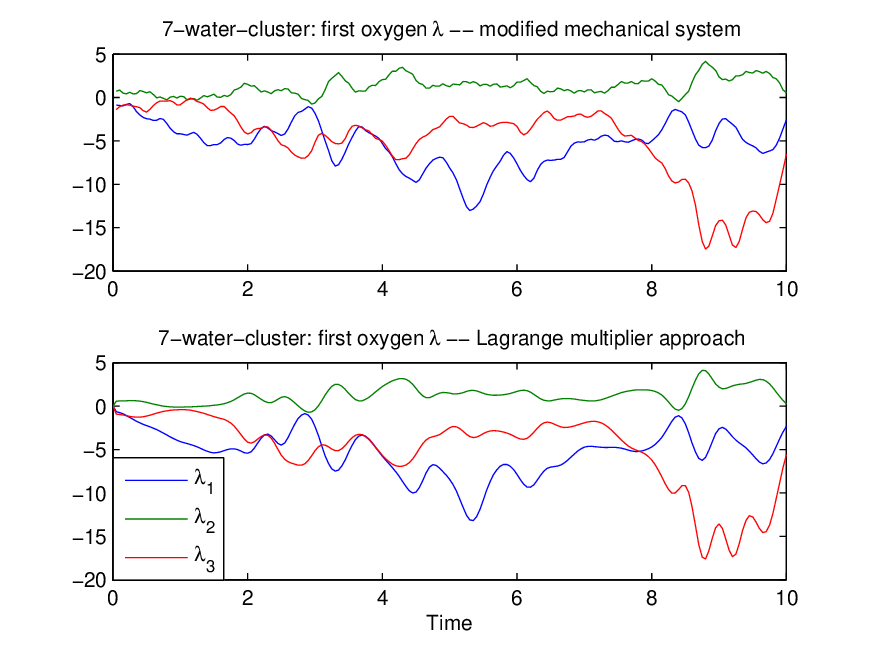}
\vspace{-6pt}
\caption{\footnotesize Lagrange multipliers from the penalized system (simplified Z\&Ss) and by SHAKE. Only 1st oxygen atom is shown, and illustration is terminated before chaos. Empirical average uses a time window of width 0.2. $\omega$ temporarily enlarged to $500$ for clearer visualization of details.}
\label{AsymptoticLambdaWater}
\end{figure}

\begin{figure}[htb]
\center
\includegraphics[width=0.8\textwidth]{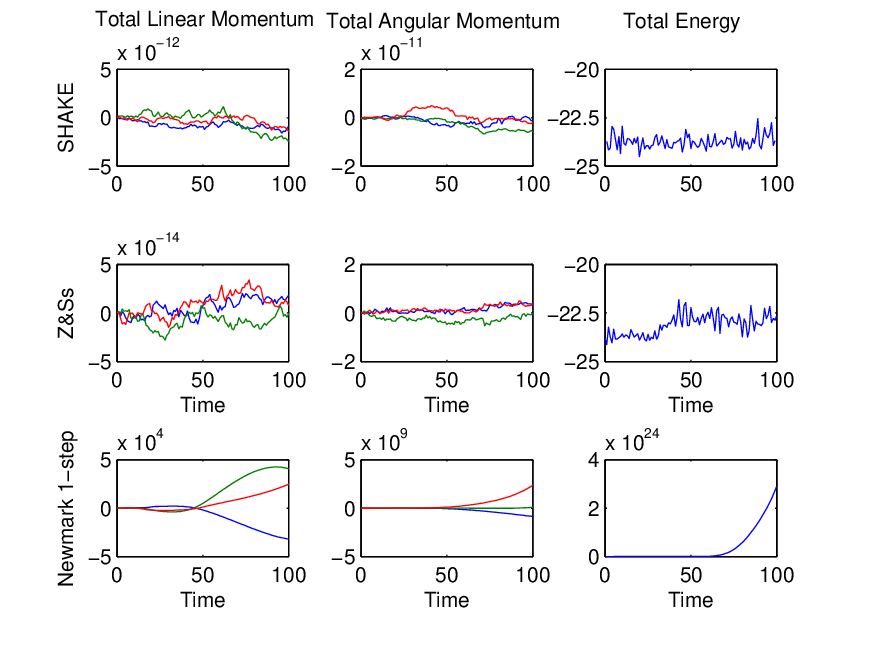}
\vspace{-8pt}
\caption{\footnotesize Energy, linear and angular momentum preservation by simplified Z\&Ss, SHAKE, and partially-solved Newmark. $h=0.05$ for all. For clarity plots are drawn with a 20:1 downsample rate.}
\label{Water7SyLiPNnearSymplecticity}
\end{figure}

We also provide two deterministic simulations (with noise and friction turned off; other parameters remains unchanged unless indicated otherwise): (i) Figure \ref{AsymptoticLambdaWater} compares Lagrange multipliers computed by SHAKE and from the penalized system to illustrate Theorem \ref{asymptoticLambdaTheorem}. (ii) Figure \ref{Water7SyLiPNnearSymplecticity} compares simplified Z\&Ss, SHAKE, and partially solved Newmark (non-symplectic) in terms of energy and momentum conservations. Simplified Z\&Ss lost symplecticity due to simplification, but it still exhibits improved preservation properties comparing to partially solved Newmark.

\begin{figure}[htb]
\includegraphics[width=\textwidth]{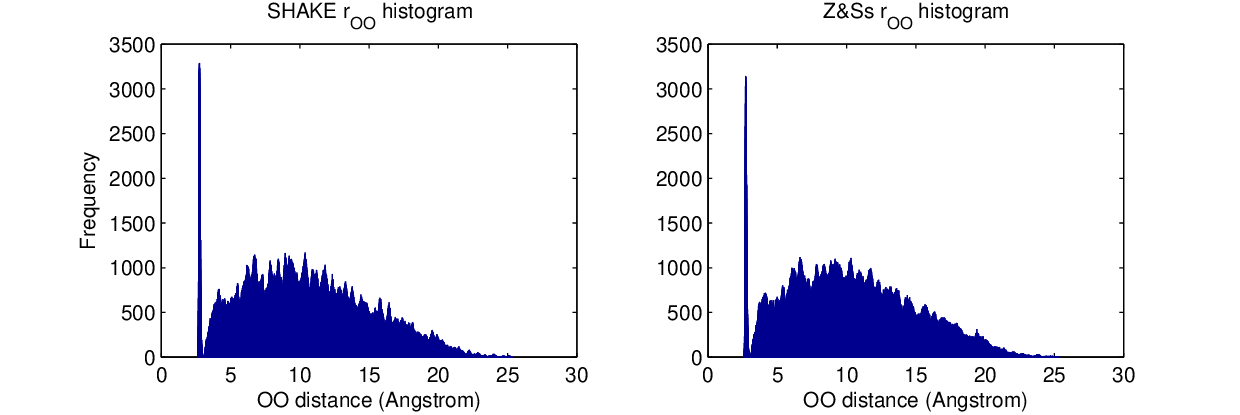}
\caption{\footnotesize Empirical OO radial distribution in 100-water cluster obtained by simulations of SHAKE and simplified Z\&Ss till $T=1000$.}
\label{Water100OO}
\end{figure}

To test scalability, we increase $N$ to $100$ ($900$ degrees of freedom) and illustrate results in Figure \ref{Water100OO}. SHAKE spent 42234 secs, including 14272 secs on solving nonlinear systems, and $\sim$28000 secs on computing $V$ and $\nabla V$, whereas simplified Z\&Ss spent 30059 secs, including 319 secs on solving linear systems, and $\sim$29000 secs on $V$, $\nabla V$ and $\text{Hess}\, V$.

\paragraph{Efficiency of force evaluation:} Although simplified Z\&Ss accelerates updates by linearization, for large systems the computational bottleneck is likely to be on force evaluations but not updates. Fortunately, significant progress has been made to accelerate force evaluations, such as the fast multipole method \cite{GrRo87}, or simply the idea of ignoring weak long-range forces. We did not employ any of them, but they can be used in adjunct to simplified Z\&Ss.

Times spent on force evaluations by SHAKE and simplified Z\&Ss are comparable; Hessian computations in simplified Z\&Ss didn't incur much overhead. This is because the potential is a function of relative distances $r_{ij}=\|x_i-x_j\|$. For such $f$,
\begin{equation}
    \frac{\partial^2 f(r)}{\partial x_i \partial x_j}=\frac{\partial r}{\partial x_i} \frac{\partial^2 f}{\partial r^2}\frac{\partial r}{\partial x_j}+\frac{\partial f}{\partial r} \frac{\partial^2 r}{\partial x_i \partial x_j},
    \label{afdjsn1gjlrnflqjadga}
\end{equation}
but $\frac{\partial r}{\partial x}$ and $\frac{\partial f}{\partial r}$ are already computed when calculating the gradient, $\frac{\partial^2 r}{\partial x_i \partial x_j}$ is cheap to obtain, and $\frac{\partial^2 f}{\partial r^2}$ is the only new component of computation but it is a scalar. In addition, nonlinear solver (e.g., Newton) in Lagrange multiplier or generalized coordinate methods requires the Hessian too because the equation to solve involves $\nabla V$.

The linear system associated with the Hessian can also be solved in $\mathcal{O}(N)$ time. This is because the Hessian is dominated by block diagonal due to localized stiff penalty terms in \eqref{pvqreuphp3410hrp1ubo}. Simplified Z\&Ss further reduces the Hessian to completely block-diagonal, and linear solves are executed molecule by molecule. Similar efficiency can be obtained for polymers as long as the number of bonds is at the same order as the number of atoms.




\subsection{Coarse time-stepping of a DNA model}
\label{SectionDNAexample}

We now show how Z\&S accelerates the simulation of an \textbf{unconstrained} multiscale system. Consider the simple DNA model proposed in \cite{Me06} and further studied, e.g., in \cite{DuMeMa09, KoOw12}. The displacement angle of the $k^\text{th}$ base in one strand, $\theta_k$, follows
\begin{equation}
    \ddot{\theta}_k=\theta_{k+1}-2\theta_{k}+\theta_{k-1}-\epsilon U'(\theta_k) ,
    \label{eq_DNA}
\end{equation}
where $U(\theta)=(\exp(-a [1-\cos(\theta)-x_0])-1)^2$ is a Morse potential modeling complementary base pairings between two DNA strands, and linear force models the tendency of alignment between neighboring bases. Unitless parameters are $a=7$, $x_0=0.3$, $\epsilon=1/1400$, and the number of base-pairs $N=200$ \cite{DuMeMa09}. Two stable configurations are given by minima of $U$ and correspond to closed double strands. Nonlinearity in this system is critical, for it leads to  transitions between metastable states that correspond to the opening of double strands. We simulate such transitions with initial positions $\theta_k=0.8+0.1\xi_k$ ($\xi_k$ i.i.d. standard norm), which is near a stable configuration, and initial momenta $\dot{\theta}_k=\cos(4\pi k/N)/\sqrt{N}$, which facilitates the opening-up \cite{DuMeMa09}.

\begin{figure}[htb]
\includegraphics[width=\textwidth]{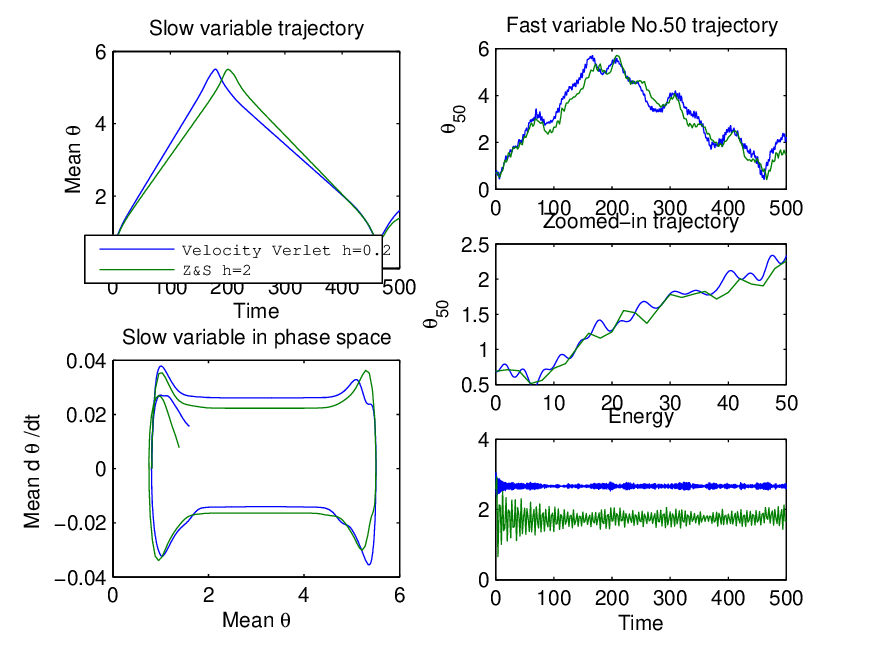}
\caption{\footnotesize DNA ($N=200$ base-pairs) conformational transitions by Z\&S and Velocity-Verlet.}
\label{fig_DNA}
\end{figure}

It is known \cite{DuMeMa09} that $\bar{\theta}=\sum \theta_k/N$ is a slow variable and can be used as a reaction parameter, whereas individual $\theta_k$'s are fast variables. Figure \ref{fig_DNA} presents simulations by Velocity-Verlet (benchmark) and Z\&S ($\beta=0.3$) in these variables. The phase portrait shows that the DNA transits between  meta-stable configurations $\theta=\text{arccos}(0.7) \approx 0.795$ and $\theta=2\pi-\text{arccos}(0.7) \approx 5.49$. Note these are long time simulations and the system is chaotic \cite{Me06}.

The Z\&S energy is lower than the benchmark because fast oscillations are damped by large time-steps.
Both $h=0.2$ in Velocity-Verlet and $h=2$ in Z\&S are near stability limits. The methods respectively used 23.26 and 3.24 seconds CPU time.


\subsection{Lie group integration}
\label{sec_LieGroup}
Formulating Z\&S as a variational principle allows us to generalize the method to mechanical systems on Lie groups.

Consider a prototypical  example of magnetized 3D rigid body with identity inertia matrix immersed in a constant magnetic field. The configuration space is $Q=\text{SO}(3)$. Denote by $R(t)\in Q$ the (generalized) rigid body position; in coordinates it is a $3\times 3$ matrix satisfying $R^T R=I$. Suppose when $R=I$ both the magnetic field and the dipole are in $z$-direction, then the potential energy can be written as $V(R)=\langle \mu R e_3, B e_3 \rangle=B\mu e_3^T R^{-1} e_3$, where $B$ and $\mu$ are field strength and dipole moment, and $e_3=\begin{bmatrix} 0 & 0 & 1 \end{bmatrix}^T$.

Let $\Omega(t)\in \mathbb{R}^3$ be convective angular velocity of the body, then the kinetic energy is $\frac{1}{2} \Omega^T \Omega$. Introduce an isomorphism between $\mathbb{R}^3$ and $\mathfrak{so}(3)$ (the Lie algebra of $\text{SO}(3)$) by
\[
    \Omega \mapsto \hat{\Omega}=\begin{bmatrix} 0 & -\Omega_3 & \Omega_2 \\ \Omega_3 & 0 & -\Omega_1 \\ -\Omega_2 & \Omega_1 & 0 \end{bmatrix}.
\]
Then $\dot{R}=R\hat{\Omega}$. It is known \cite{MaRa10} that dynamics of this mechanical system  can be obtained from either of the following equivalent variational principles:
\begin{itemize}
\item
    \begin{equation}
        \delta \int_0^T L(R,\dot{R}) dt=0 ,
        \label{eq_LieVarPrinciple}
    \end{equation}
    with arbitrary variations of $R(t)\in Q$.
\item
    \begin{equation}
        \delta \int_0^T l(R,\xi) dt=0 ,
        \label{eq_reducedLieVarPrinciple}
    \end{equation}
    with variations in the form $\delta\xi=\dot{\eta}+\text{ad}_\xi \eta$ under $R\in Q$ and $\xi=R^{-1}\dot{R}$.
\end{itemize}
For our system, $L(R,\dot{R})=\frac{1}{4} \text{tr}(\dot{R}^T\dot{R})-B\mu e_3^T R^{-1} e_3$ and $l(R,\hat{\Omega})=\frac{1}{2}\Omega^T \Omega-B\mu e_3^T R^{-1} e_3$ (note $\Omega^T \Omega = \frac{1}{2}\text{tr}(\hat{\Omega}^T\hat{\Omega})$).

We propose to simulate the system by modifying \eqref{eq_LieVarPrinciple}. The result is compared to a benchmark derived from \eqref{eq_reducedLieVarPrinciple} via Hamilton-Pontryagin princple, backward Variational Euler discretization, and Cayley approximation of the exponential map (see \cite{iserles2001cayley,Hairer06} for Cayley approximation, \cite{BoMa2009} for the benchmark method, and \cite{iserles2000lie,lee2007lie} for examples of other Lie group integrators). The benchmark uses update rules:
\begin{equation}
\begin{cases}
    R_{k+1} &= R_k \left(I-h \hat{\Omega}_{k+1}/2 \right)^{-1} \left(I+h \hat{\Omega}_{k+1}/2 \right) \\
    \hat{\Omega}_{k+1} &= \hat{\Omega}_k + \frac{h^2}{4} \left( \hat{\Omega}_{k+1}^T \hat{\Omega}_{k+1} \hat{\Omega}_{k+1}^T - \hat{\Omega}_{k}^T \hat{\Omega}_{k} \hat{\Omega}_{k}^T \right) + h R_{k}\frac{\partial l}{\partial R}(R_k,\hat{\Omega}_{k+1}).
\end{cases}
\label{eq_VarLie}
\end{equation}
Note $R$ is in a 3-dimensional manifold, and the differential in the last term shouldn't be computed as partial derivative with respect to 9 Cartesian coordinates of $R$, otherwise the last term won't be in $\mathfrak{so}(3)$. Instead, we follow \cite{holm1998euler} and obtain
\begin{equation}
    h R_{k}\frac{\partial l}{\partial R}(R_k,\hat{\Omega}_{k+1}) = h B\mu R_k\frac{-e_3^T R_k^{-1} e_3}{\partial R_k} = h B\mu {\left( (R_k^{-1} e_3) \times e_3 \right)}^\wedge.
    \label{eq_DerivativeOnLieGroup}
\end{equation}
\eqref{eq_VarLie} is variational and thus numerically energy and momentum preserving. Thanks to Cayley approximation, it also preserves the SO(3) structure in the sense that $R_k^T R_k=I$ up to arithmetic error. However, variational methods of this type are intrinsically nonlinearly implicit due to curved geometry when $Q$ is noncommutative (e.g., \cite{BoMa2009,kobilarov2009lie}).

Our goal is to avoid expensive nonlinear solves and bypass force evaluations that require geometric calculations (such as \eqref{eq_DerivativeOnLieGroup}). To do so, we first add penalization to \eqref{eq_LieVarPrinciple}:
\[
    \delta \int_0^T \frac{1}{4} \text{tr} (\dot{R}^T \dot{R}) - B\mu R(3,3) + \frac{1}{2}\omega^2 \text{tr}((R^T R-I)^T (R^T R-I)) dt=0,
\]
where $R\in \text{SO}(3)$ is relaxed to $R\in \mathbb{R}^{3\times 3}$. We then discretize the action as follows:
\begin{align}
    & \mathcal{L}_d(R_k,R_{k+1},a_k) = \text{tr} \left( \frac{1}{4} \left( \frac{R_{k+1}-R_k}{h} \right)^T \left( \frac{R_{k+1}-R_k}{h} \right) - \beta h^2 \frac{1}{2} a_k^T a_k \right) \nonumber\\
    & \qquad - B\mu e_3^T(R_k+\beta h^2 a_k)e_3 - \text{tr} \left( \frac{1}{2}\omega^2 \left(\left((R_k+\beta h^2 a_k)^T (R_k+\beta h^2 a_k)-I\right)^2 \right) \right)
    \label{eq_DiscreteAction1stOrderLieGroup}
\end{align}
Finally, we truncate terms that are higher than 2nd-order in $a_k$. After using trace identities $\text{tr}(A B)=\text{tr}(B A)$ and $\text{tr}(A^T)=\text{tr}(A)$, the truncated action simplifies to
\begin{align*}
    & \tilde{\mathcal{L}}_d(R_k,R_{k+1},a_k) = \text{tr} \left( \frac{1}{4} \left( \frac{R_{k+1}-R_k}{h} \right)^T \left( \frac{R_{k+1}-R_k}{h} \right) - \beta h^2 \frac{1}{2} a_k^T a_k \right) \\
    & \qquad - B\mu e_3^T(R_k+\beta h^2 a_k)e_3 - \text{tr} \Bigg( \frac{1}{2}\omega^2 \Big(
    (R_k^T R_k-I)^2 + 4\beta h^2 (R_k^T R_k R_k^T-R_k^T) a_k \\
    & \qquad\qquad\qquad\qquad + 2\beta^2 h^4 (a_k^T a_k R_k^T R_k + R_k^T a_k R_k^T a_k + a_k^T R_k R_k^T a_k - a_k^T a_k )
    \Big) \Bigg).
\end{align*}
Unconstrained variation of this action with respect to $a_k$ gives
\begin{equation}
    a_k = -B\mu e_3 e_3^T -2\omega^2 \left( R_k R_k^T R_k - R_k + \beta h^2 (a_k R_k^T R_k + R_k a_k^T R_k + R_k R_k^T a_k - a_k) \right).
    \label{eq_aInSO3}
\end{equation}
Standard variational integrator construction leads to
\[
    p_k = -D_1 \tilde{\mathcal{L}}_d (R_k,R_{k+1},a_k), \qquad
    p_{k+1} = D_2 \tilde{\mathcal{L}}_d (R_k,R_{k+1},a_k).
\]
Let $f_k = a_k + 2\omega^2 \beta^2 h^4 \left( a_k^T a_k R_k^T + a_k^T R_k a_k^T + R_k^T a_k a_k^T \right)$ and use \eqref{eq_aInSO3} for simplification, then the above becomes
\[ \begin{cases}
    p_{k+1} &= p_k + h f_k \\
    R_{k+1} &= R_k + 2h p_{k+1}
\end{cases}. \]
These are our variational linearized SO(3) integrator.
Note \eqref{eq_DiscreteAction1stOrderLieGroup} is based on a 1st-order quadrature; 2nd-order trapezoidal rule would lead to
\[ \begin{cases}
    p_{k+1/2} &= p_k + \frac{h}{2} f_k \\
    R_{k+1} &= R_k + 2h p_{k+1} \\
    p_{k+1} &= p_k + \frac{h}{2} f_{k+1}
\end{cases}. \]
These are similar to Z\&S updates (Integrator \ref{SyLiPN}) although Z\&S works in $\mathbb{R}^n$. Some may question the usefulness of a variational formulation, because one can vectorize $R$ into $9$-dimension, view the penalized system as Newton's equation, and then use Z\&S. In Z\&S updates \eqref{SyLiPNeq}, however, $V^{(3)}$ is essentially a 6-tensor, and its brute-force calculation in coordinates, as well as its contractions with $a$ from both left and right, will be unpleasant. A variational approach minimizes the involvement of coordinates and reduces the effort.

\begin{figure}
\vspace{-15pt}
\includegraphics[width=\textwidth]{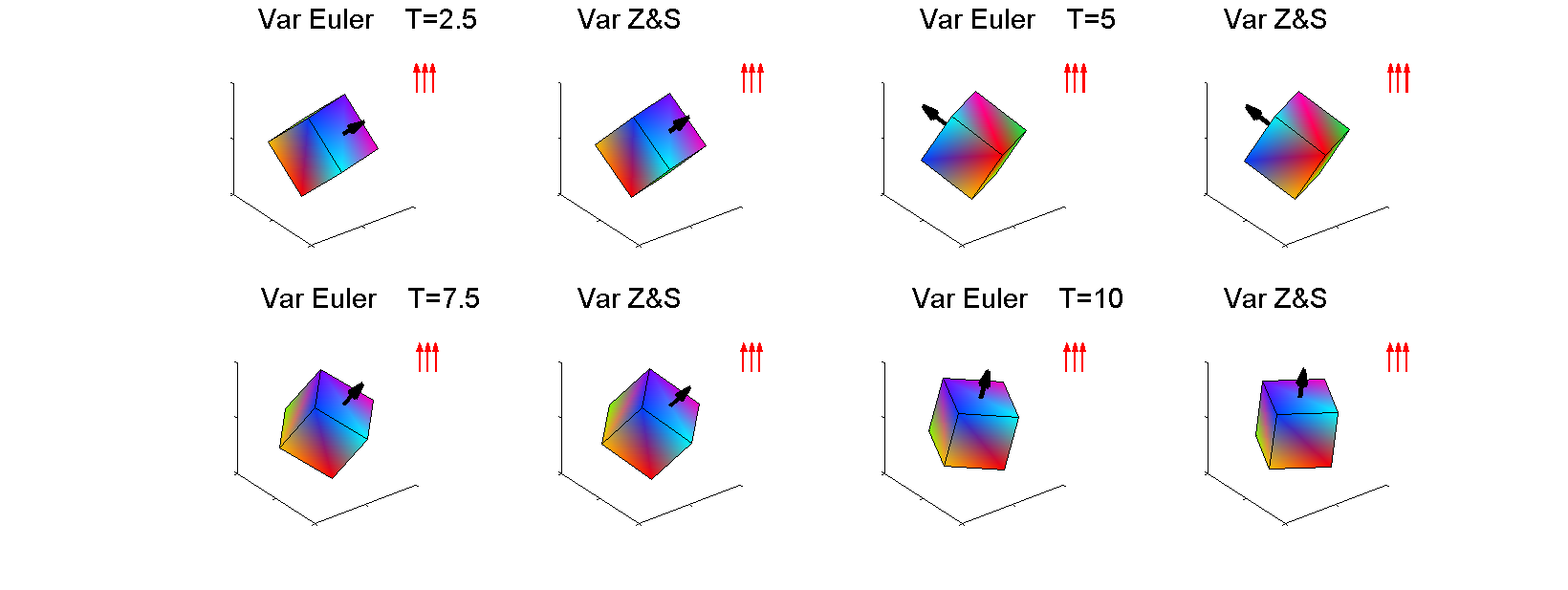}
\vspace{-20pt}
\caption{\footnotesize Snapshots of magnetized rigid body dynamics. Red and black arrows represent magnetic field and dipole.}
\label{fig_Lie_SyLiPN_hd1}
\end{figure}

Figure \ref{fig_Lie_SyLiPN_hd1} demonstrates benchmark ($h=0.0001$) and variational Z\&S simulations ($h=0.1$, $\omega=10$, $\beta=0.4$); their difference, regarded as our method's error, is quantified in Figure \ref{fig_Lie_SyLiPN_hd1_err}; error of the benchmark method with $h=0.1$ is also provided in Figure \ref{fig_Lie_VarEuler_hd1_err} as a comparison. The full simulation is available at \url{http://youtu.be/29deMRDRsuU}. Initial conditions are $R(0)=I$ and $\Omega(0)=[1;0.2;0.1]$.

Our method and variational Euler with both $h=0.1$ respectively spent 0.03 and 1.45 seconds on computations. However, variational Lie group integrator is much better at preserving the Lie group structure. Applicabilities of the two approaches are disjoint: for example, variational Z\&S generally suits computer graphics better, where real time rendering requires high efficiency, while demand on accuracy is moderate (as long as the result looks good); to orient satellites (e.g., \cite{junge2005optimal}), on the other hand, one should choose variational Lie group integrators over variational Z\&S, and it is worth CPU hours to precompute trajectories with high fidelity.

\begin{figure}
\centering
\footnotesize
\subfigure[Variational Z\&S ($h=0.1$)]{
\includegraphics[width=0.48\textwidth]{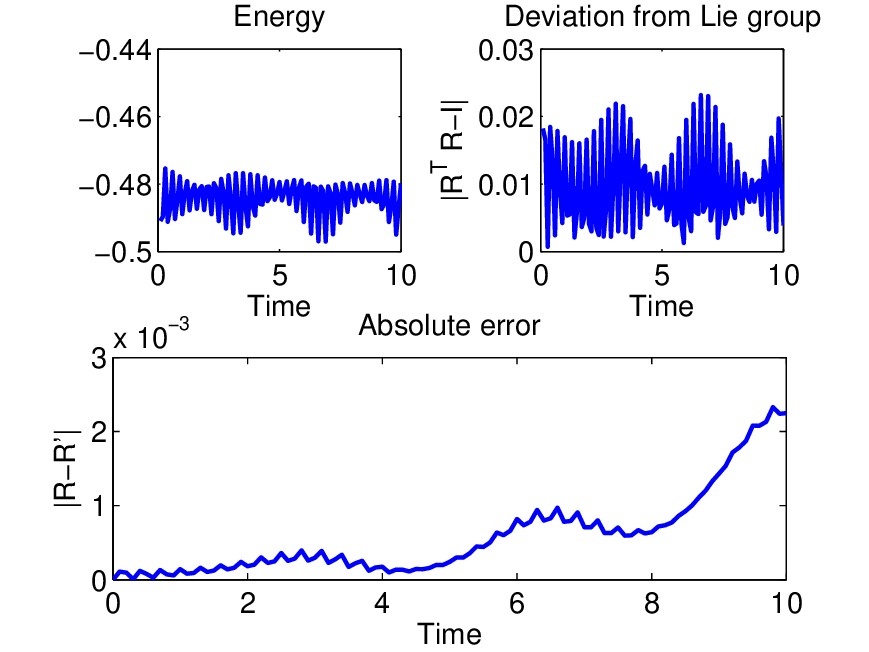}
\label{fig_Lie_SyLiPN_hd1_err}
}
\subfigure[Variational Euler on Lie group ($h=0.1$)]{
\includegraphics[width=0.48\textwidth]{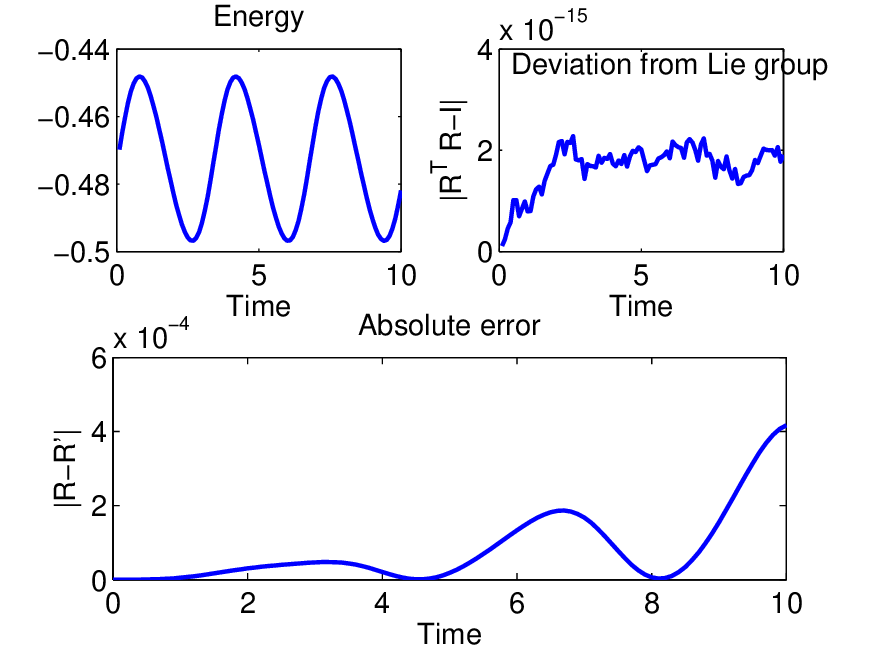}
\label{fig_Lie_VarEuler_hd1_err}
}
\vspace{-15pt}
\caption{\footnotesize Preservation of energy, Lie group structure, and deviation from benchmark trajectory $R'$.}
\end{figure}

\section{Acknowledgement}
This work was supported by NSF grant CMMI-092600, a generous gift from UTRC, and Courant Instructorship from New York University. We thank Mathieu Desbrun for motivation and discussions, Joel Tropp and Eitan Grinspun for discussions, Carmen Sirois for proofreading the manuscript, and anonymous reviewers for comments.

\bibliographystyle{siam}
\footnotesize
\bibliography{molei21}

\def\cprime{$'$} \def\cprime{$'$} \def\cprime{$'$}
  \def\cydot{\leavevmode\raise.4ex\hbox{.}}
\begin{thebibliography}{10}

\bibitem{AbMa08}
{\sc R.~Abraham and J.~E. Marsden}, {\em Foundations of Mechanics}, American
  Mathematical Society, 2nd~ed., 2008.

\bibitem{An83}
{\sc H.~C. Andersen}, {\em Rattle: {A ``velocity''} version of the shake
  algorithm for molecular dynamics calculations}, J. Comput. Phys, 52 (1983),
  pp.~24--34.

\bibitem{BaFrLe11}
{\sc J.~Bajars, J.~Frank, and B.~Leimkuhler}, {\em Stochastic-dynamical
  thermostats for constraints and stiff restraints}, The European Physical
  Journal Special Topics, 200 (2011), pp.~131--152.

\bibitem{BeWa76}
{\sc R.~M. Beam and R.~Warming}, {\em An implicit finite-difference algorithm
  for hyperbolic systems in conservation-law form}, Journal of Computational
  Physics, 22 (1976), pp.~87 -- 110.

\bibitem{Bo98}
{\sc F.~Bornemann}, {\em Homogenization in time of singularly perturbed
  mechanical systems}, vol.~1687 of Lecture Notes in Mathematics, Springer,
  Berlin, Heidelberg, New York, 1998.

\bibitem{BoSc95}
{\sc F.~A. Bornemann and C.~Sch{\"u}tte}, {\em A mathematical approach to
  smoothed molecular dynamics: Correcting potentials for freezing bond angles},
  Konrad-Zuse-Zentrum f{\"u}r Informationstechnik Berlin, 1995.

\bibitem{BoSc97}
\leavevmode\vrule height 2pt depth -1.6pt width 23pt, {\em Homogenization of
  hamiltonian systems with a strong constraining potential}, Physica D:
  Nonlinear Phenomena, 102 (1997), pp.~57--77.

\bibitem{BoMa2009}
{\sc N.~Bou-Rabee and J.~E. Marsden}, {\em Hamilton--{P}ontryagin integrators
  on {L}ie groups part {I}: {I}ntroduction and structure-preserving
  properties}, Foundations of Computational Mathematics, 9 (2009),
  pp.~197--219.

\bibitem{BoOw:09}
{\sc N.~Bou-Rabee and H.~Owhadi}, {\em Long-run accuracy of variational
  integrators in the stochastic context}, SIAM J. Numer. Anal., 48 (2010),
  pp.~278--297.

\bibitem{ChSc90}
{\sc P.~J. Channell and C.~Scovel}, {\em Symplectic integration of hamiltonian
  systems}, Nonlinearity, 3 (1990), p.~231.

\bibitem{ChKa02}
{\sc F.~Chiba and T.~Kako}, {\em Newmark's method and discrete energy applied
  to resistive mhd equation}, Vietnam J. Math., 30 (2002), pp.~501--520.

\bibitem{ErBoBu02}
{\sc S.~Erlicher, L.~Bonaventura, and O.~S. Bursi}, {\em The analysis of the
  {G}eneralized $\alpha$ method for non-linear dynamic problems}, Comput.
  Mech., 28 (2002), pp.~83--104.

\bibitem{FaGr11}
{\sc E.~Faou and B.~Grebert}, {\em Hamiltonian interpolation of splitting
  approximations for nonlinear {PDE}s}, Foundations of Computational
  Mathematics, 11 (2011), pp.~381--415.

\bibitem{FiJi10}
{\sc F.~Filbet and S.~Jin}, {\em A class of asymptotic preserving schemes for
  kinetic equations and related problems with stiff sources},  (2010).
\newblock arXiv:0905.1378. Accepted by J. Comput. Phys.

\bibitem{Fixman:74}
{\sc M.~Fixman}, {\em Classical statistical mechanics of constraints: A theorem
  and application to polymers}, Proc. Nat. Acad. Sci. USA, 71-8 (1974),
  pp.~3050--3053.

\bibitem{Skeel:99}
{\sc B.~Garc\'{i}a-Archilla, J.~M. Sanz-Serna, and R.~D. Skeel}, {\em
  Long-time-step methods for oscillatory differential equations}, SIAM J. Sci.
  Comput., 20 (1999), pp.~930--963.

\bibitem{GrRo87}
{\sc L.~F. Greengard and V.~Rokhlin}, {\em A fast algorithm for particle
  simulations}, J. Comput. Phys, 73 (1987), pp.~325--348.

\bibitem{Hairer06}
{\sc E.~Hairer, C.~Lubich, and G.~Wanner}, {\em Geometric Numerical
  Integration: Structure-Preserving Algorithms for Ordinary Differential
  Equations}, Springer, Berlin Heidelberg New York, second~ed., 2006.

\bibitem{HaWa96}
{\sc E.~Hairer and G.~Wanner}, {\em Solving ordinary differential equations
  II}, Springer, 2nd~ed., 1996.

\bibitem{Ha08}
{\sc C.~Hartmann}, {\em An ergodic sampling scheme for constrained hamiltonian
  systems with applications to molecular dynamics}, Journal of Statistical
  Physics, 130 (2008), pp.~687--711.

\bibitem{HeBeBeFr97}
{\sc B.~Hess, H.~Bekker, H.~J.~C. Berendsen, and J.~G. E.~M. Fraaije}, {\em
  {LINCS}: {A} linear constraint solver for molecular simulations}, J. Comput.
  Chem., 18 (1997), pp.~1463--1472.

\bibitem{HeGoGo07}
{\sc J.~S. Hesthaven, S.~Gottlieb, and D.~Gottlieb}, {\em Spectral Methods for
  Time-Dependent Problems}, vol.~21 of Cambridge Monographs on Applied and
  Computational Mathematics, Cambridge University Press, United Kingdom, 2007.

\bibitem{holm1998euler}
{\sc D.~D. Holm, J.~E. Marsden, and T.~S. Ratiu}, {\em The euler--poincar{\'e}
  equations and semidirect products with applications to continuum theories},
  Advances in Mathematics, 137 (1998), pp.~1--81.

\bibitem{Hu77}
{\sc T.~J.~R. Hughes}, {\em A note on the stability of newmark's algorithm in
  nonlinear structural dynamics}, Int. J. Numer. Meth. Eng., 11 (1977),
  pp.~383--386.

\bibitem{iserles2001cayley}
{\sc A.~Iserles}, {\em On {C}ayley-transform methods for the discretization of
  {L}ie-group equations}, Found. Comput. Math., 1 (2001), pp.~129--160.

\bibitem{iserles2000lie}
{\sc A.~Iserles, H.~Z. Munthe-Kaas, S.~P. N{\o}rsett, and A.~Zanna}, {\em
  {L}ie-group methods}, Acta Numerica 2000, 9 (2000), pp.~215--365.

\bibitem{JaVaRo93}
{\sc A.~Jain, N.~Vaidehi, and G.~Rodriguez}, {\em A fast recursive algorithm
  for molecular dynamics simulation}, J. Comput. Phys, 106 (1993),
  pp.~258--268.

\bibitem{JoCh83}
{\sc W.~L. Jorgensen, J.~Chandrasekhar, J.~D. Madura, R.~W. Impey, and M.~L.
  Klein}, {\em Comparison of simple potential functions for simulating liquid
  water}, J. Chem. Phys., 79 (1983), p.~926.

\bibitem{junge2005optimal}
{\sc O.~Junge and S.~Ober-Blobaum}, {\em Optimal reconfiguration of formation
  flying satellites}, in Decision and Control, 2005 and 2005 European Control
  Conference. CDC-ECC'05. 44th IEEE Conference on, IEEE, 2005, pp.~66--71.

\bibitem{KaMaOrWe00}
{\sc C.~Kane, J.~E. Marsden, M.~Ortiz, and M.~West}, {\em Variational
  integrators and the {N}ewmark algorithm for conservative and dissipative
  mechanical systems}, Int. J. Numer. Meth. Eng., 49 (2000), pp.~1295--1325.

\bibitem{KeCo96}
{\sc J.~Kevorkian and J.~D. Cole}, {\em Multiple scale and singular
  perturbation methods}, vol.~114 of Applied Mathematical Sciences,
  Springer-Verlag, New York, 1996.

\bibitem{kobilarov2009lie}
{\sc M.~Kobilarov, K.~Crane, and M.~Desbrun}, {\em Lie group integrators for
  animation and control of vehicles}, ACM Transactions on Graphics, 28 (2009),
  p.~16.

\bibitem{KoOw12}
{\sc W.~Koon, H.~Owhadi, M.~Tao, and T.~Yanao}, {\em Control of a model of
  {DNA} division via parametric resonance}.
\newblock arXiv:1211.4064. Submitted, 2012.

\bibitem{KrGu01}
{\sc V.~Kr\"{a}utler, W.~F. van Gunsteren, and P.~H. H¨¹nenberger}, {\em A fast
  {SHAKE} algorithm to solve distance constraint equations for small molecules
  in molecular dynamics simulations}, J. Comput. Chem., 22 (2001),
  pp.~501--508.

\bibitem{KuRa96}
{\sc D.~Kuhl and E.~Ramm}, {\em Constraint energy momentum algorithm and its
  application to non-linear dynamics of shells}, Computer Methods in Applied
  Mechanics and Engineering, 136 (1996), pp.~293 -- 315.

\bibitem{lee2007lie}
{\sc T.~Lee, M.~Leok, and N.~H. McClamroch}, {\em Lie group variational
  integrators for the full body problem}, Computer Methods in Applied Mechanics
  and Engineering, 196 (2007), pp.~2907--2924.

\bibitem{MR2132573}
{\sc B.~Leimkuhler and S.~Reich}, {\em Simulating {H}amiltonian dynamics},
  vol.~14 of Cambridge Monographs on Applied and Computational Mathematics,
  Cambridge University Press, Cambridge, 2004.

\bibitem{LeRoSt12}
{\sc T.~Leli{\`e}vre, M.~Rousset, and G.~Stoltz}, {\em Langevin dynamics with
  constraints and computation of free energy differences}, Mathematics of
  Computation, 81 (2012), pp.~2071--2125.

\bibitem{LiAbE:08}
{\sc T.~Li, A.~Abdulle, and W.~E}, {\em Effectiveness of implicit methods for
  stiff stochastic differential equations}, Commun. Comput. Phys., 3 (2008),
  pp.~295--307.

\bibitem{MaRa10}
{\sc J.~E. Marsden and T.~S. Ratiu}, {\em Introduction to Mechanics and
  Symmetry}, Springer, 2nd~ed., 2010.

\bibitem{MaWe:01}
{\sc J.~E. Marsden and M.~West}, {\em Discrete mechanics and variational
  integrators}, Acta Numerica,  (2001), pp.~357--514.

\bibitem{McAt92}
{\sc R.~I. McLachlan and P.~Atela}, {\em The accuracy of symplectic
  integrators}, Nonlinearity, 5 (1992), p.~541.

\bibitem{Me06}
{\sc I.~Mezi\'{c}}, {\em On the dynamics of molecular conformation}, Proc.
  Natl. Acad. Sci., 103 (2006), pp.~7542--7547.

\bibitem{MiKo92}
{\sc S.~Miyamoto and P.~A. Kollman}, {\em Settle: {A}n analytical version of
  the {SHAKE} and {RATTLE} algorithm for rigid water models}, J. Comput. Chem.,
  13 (1992), pp.~952--962.

\bibitem{Neri:88}
{\sc F.~Neri}, {\em Lie algebras and canonical integration}, tech. report,
  Department of Physics, University of Maryland, 1988.

\bibitem{Ne59}
{\sc N.~M. Newmark}, {\em A method of computation for structural dynamics},
  Proc. ASCE, 85 (1959), pp.~67--94.

\bibitem{PeSkYa85}
{\sc D.~Perchak, J.~Skolnick, and R.~Yaris}, {\em Dynamics of rigid and
  flexible constraints for polymers. {E}ffect of the {F}ixman potential},
  Macromolecules, 18 (1985), pp.~519--525.

\bibitem{PlBa88}
{\sc J.~C. Platt and A.~H. Barr}, {\em Constraints methods for flexible
  models}, SIGGRAPH Comput. Graph., 22 (1988), pp.~279--288.

\bibitem{Re95}
{\sc S.~Reich}, {\em Smoothed dynamics of highly oscillatory {H}amiltonian
  systems}, Physica D: Nonlinear Phenomena, 89 (1995), pp.~28--42.

\bibitem{RiSc84}
{\sc P.~H. Richter and H.-J. Scholz}, {\em Stochastic phenomena and chaotic
  behaviour in complex systems}, Springer-Verlag, Berlin and New York, 1984.

\bibitem{RuUn57}
{\sc H.~Rubin and P.~Ungar}, {\em Motion under a strong constraining force},
  Communications on Pure and Applied Mathematics, 10 (1957), pp.~65--87.

\bibitem{RyCiBe97}
{\sc J.-P. Ryckaert, G.~Ciccotti, and H.~J.~C. Berendsen}, {\em Numerical
  integration of the cartesian equations of motion of a system with
  constraints: molecular dynamics of n-alkanes}, J. Comput. Phys, 23 (1977),
  pp.~327--341.

\bibitem{Sanz-Serna:08}
{\sc J.~M. Sanz-Serna}, {\em Mollified impulse methods for highly oscillatory
  differential equations}, SIAM J. Numer. Anal., 46 (2) (2008), pp.~1040--1059.

\bibitem{Sc10}
{\sc T.~Schlick}, {\em Molecular Modeling and Simulation}, Springer, 2nd~ed.,
  2010.

\bibitem{ScBo97homogenization}
{\sc C.~Sch{\"u}tte and F.~A. Bornemann}, {\em Homogenization approach to
  smoothed molecular dynamics}, Nonlinear analysis, theory, methods \&
  applications, 30 (1997), pp.~1805--1814.

\bibitem{ShZe02}
{\sc J.~Shatah and C.~Zeng}, {\em Periodic solutions for {H}amiltonian systems
  under strong constraining forces}, Journal of Differential Equations, 186
  (2002), pp.~572--585.

\bibitem{SkSr00}
{\sc R.~D. Skeel and K.~Srinivas}, {\em Nonlinear stability analysis of
  area-preserving integrators}, SIAM J. Numer. Anal., 38 (2000), pp.~129--148.

\bibitem{SkZhSc97}
{\sc R.~D. Skeel, G.~Zhang, and T.~Schlick}, {\em A family of symplectic
  integrators: Stability, accuracy, and molecular dynamics applications}, SIAM
  J. Sci. Comput., 18 (1997), pp.~203--222.

\bibitem{Takens80}
{\sc F.~Takens}, {\em Motion under the influence of a strong constraining
  force}, in Global theory of dynamical systems, Springer, 1980, pp.~425--445.

\bibitem{SIM2}
{\sc M.~Tao, H.~Owhadi, and J.~E. Marsden}, {\em From efficient symplectic
  exponentiation of matrices to symplectic integration of high-dimensional
  {H}amiltonian systems with slowly varying quadratic stiff potentials}, Appl.
  Math. Res. Express, 2011, pp.~242--280.

\bibitem{FLAVOR10}
\leavevmode\vrule height 2pt depth -1.6pt width 23pt, {\em Nonintrusive and
  structure preserving multiscale integration of stiff {ODEs}, {SDEs} and
  {Hamiltonian} systems with hidden slow dynamics via flow averaging},
  Multiscale Model. Simul., 8 (2010), pp.~1269--1324.

\bibitem{TePl87}
{\sc D.~Terzopoulos, J.~Platt, A.~Barr, and K.~Fleischer}, {\em Elastically
  deformable models}, SIGGRAPH Comput. Graph., 21 (1987), pp.~205--214.

\bibitem{DuMeMa09}
{\sc P.~D. Toit, I.~Mezi\'{c}, and J.~Marsden}, {\em Coupled oscillator models
  with no scale separation}, Phys. D, 238 (2009), pp.~490--501.

\bibitem{SpMa98}
{\sc D.~van~der Spoel, P.~J. van Maaren, and H.~J.~C. Berendsen}, {\em A
  systematic study of water models for molecular simulation: {D}erivation of
  water models optimized for use with a reaction field}, J. Chem. Phys., 108
  (1998), p.~10220.

\bibitem{VaCi2006}
{\sc E.~Vanden-Eijnden and G.~Ciccotti}, {\em Second-order integrators for
  {L}angevin equations with holonomic constraints}, Chem. Phys. Lett., 429
  (2006), pp.~310--316.

\bibitem{WeMa97}
{\sc J.~M. Wendlandt and J.~E. Marsden}, {\em Mechanical integrators derived
  from a discrete variational principle}, Phys. D, 106 (1997), pp.~223--246.

\bibitem{WiFl87}
{\sc A.~Witkin, K.~Fleischer, and A.~Barr}, {\em Energy constraints on
  parameterized models}, SIGGRAPH Comput. Graph., 21 (1987), pp.~225--232.

\bibitem{WoOd88}
{\sc W.~L. Wood and M.~E. Oduor}, {\em Stability properties of some algorithms
  for the solution of nonlinear dynamic vibration equations}, Commun. Appl.
  Numer. Methods, 4 (1988), pp.~205--212.

\bibitem{Yoshida:90}
{\sc H.~Yoshida}, {\em Construction of higher order symplectic integrators},
  Phys. Lett. A, 150 (1990), pp.~262--268.

\bibitem{ZhSk97}
{\sc M.~Zhang and R.~D. Skeel}, {\em Cheap implicit symplectic integrators},
  Applied Numerical Mathematics, 25 (1997), pp.~297--302.

\end{thebibliography}

\normalsize
\section{Appendix}
\begin{Lemma}
    Consider \eqref{ourConstrainedDynamics}. If $V$ is bounded from below, then there is a constant $C$, such that $\|g(q^\omega(s))\|\leq C/\omega$ for all $s$. Moreover, if $V(q)$ diverges to infinity as $|q|\rightarrow \infty$, then there is a constant $\tilde{C}$ such that $\|q^\omega(s)\|\leq \tilde{C}$.
    \label{Lemma1}
\end{Lemma}
\begin{proof}
    Note the energy $[\dot{q}^\omega]^T M\dot{q}^\omega/2+V(q^\omega)+\omega^2 g(q^\omega)^T g(q^\omega)$ in the penalized system \eqref{ourConstrainedDynamics} is conserved and determined by initial condition. Therefore, $V(\cdot)$ being bounded from below and $g^T g\geq 0$ imply that $\omega^2 g(q^\omega)^T g(q^\omega)=\mathcal{O}(1)$. Hence $g(q^\omega(s))=\mathcal{O}(1/\omega)$.

    By a similar energy argument, since $[\dot{q}^\omega]^T M\dot{q}^\omega/2\geq 0$ and $g(q^\omega)^T g(q^\omega)\geq 0$, $V(q^\omega)$ is bounded from above too, which implies that $q^\omega$ remains bounded.
\end{proof}

\begin{Lemma}
    Consider the solution to a conserved mechanical system
    \begin{equation}
    \begin{cases}
        \ddot{x}^\omega=f_1(x^\omega,y^\omega) \\
        \ddot{y}^\omega=f_2(x^\omega,y^\omega)-\omega^2 g(y^\omega)^T \nabla g(y^\omega)
        \label{gh1orh3b1o4t143t1}
    \end{cases},
    \end{equation}
    where $x^\omega$ and $y^\omega$ are vectors, and $x^\omega(0)=x_0$, $\dot{x}^\omega(0)=\dot{x}_0$, $y^\omega(0)=y_0$, $\dot{y}^\omega(0)=\dot{y}_0$.
    Suppose $f_1$, $f_2$ and $\nabla g$ are $C^1$ with bounded derivatives, $x^\omega$ and $y^\omega$ are bounded, $g(y_0)=0$ and $\frac{d}{dt}g(y_0)=0$, and $g(\cdot)$ has a non-degenerate Jacobian in a neighborhood of $y_0$, then
    \begin{equation}
        \lambda(t):=-\lim_{T\rightarrow 0} \lim_{\omega\rightarrow \infty} \frac{1}{T}\int_t^{t+T} \omega^2 g(y^\omega(s))\, ds
        \label{asymptoticLambdaFlat}
    \end{equation}
    exists and is finite. Denote by $x(t)$, $y(t)$ the solution to
    \begin{equation}
    \begin{cases}
        \ddot{x}=f_1(x,y) \\
        \ddot{y}=f_2(x,y)+\lambda^T \nabla g(y) \\
        g(y)=0
    \end{cases}
    \end{equation}
    with the same initial conditions $x_0,\dot{x}_0,y_0,\dot{y}_0$, then as $\omega\rightarrow\infty$,
    \begin{equation}
    \begin{cases}
        x^\omega \rightarrow x \\
        y^\omega \xrightarrow{F} y \\
        g(y^\omega) \rightarrow 0
    \end{cases}
    \end{equation}
    \label{Lemma2}
\end{Lemma}
\begin{proof}
    We employ the multiscale averaging framework described in \cite{FLAVOR10} to demonstrate the convergence. Here $x^\omega$ is a slow variable and its evolution corresponds to the constrained dynamics. $y^\omega$ is a fast variable corresponding to a fluctuating deviation from the constraint manifold at a characteristic timescale of $\mathcal{O}(1/\omega)$, and it lies in the normal bundle of the constraint manifold.

    First, consider the linear constraint case in which $g(y)=C y^T$ for some non-singular $C$ (the affine case can be similarly treated by shifting $y$). $y$ dynamics is governed by
    \begin{equation}
        \ddot{y}^\omega=f_2(x^\omega,y^\omega)-\omega^2 y^\omega C^TC .
        \label{g8o173obui1bf}
    \end{equation}
    This is a forced harmonic oscillator, and its solution can be written as:
    \begin{equation}
        y^\omega(t)=\int_0^t f_2(x^\omega(s),y^\omega(s)) \sin(\omega \tilde{C} s) \tilde{C}^{-1}/\omega \, ds ,
        \label{ourgibq31ovo1}
    \end{equation}
    where $\tilde{C}=\sqrt{C^T C}$ is the well-defined matrix square root, and matrix sin is defined either by Taylor expansion or diagonalization. Note there is no propagation of initial condition because $y^\omega(0)=0$.

    It can be shown from \eqref{ourgibq31ovo1} (for instance, by Lemma 3.8 in \cite{SIM2}; the idea is that an addition $1/\omega$ comes from the $\sin$ due to integration by parts) that $y^\omega(t)$ is $\mathcal{O}(\omega^{-2})$ at least up to $t=o(1)$, and $y^\omega$ is asymptotically periodic (because \eqref{g8o173obui1bf} is asymptotically linear) and hence locally ergodic on energy shell (with Dirac ergodic measure).

    Since $y^\omega$ is locally ergodic on energy shell, \eqref{asymptoticLambdaFlat} well-defines $\lambda$, and Theorem 1.2 in \cite{FLAVOR10} guarantees that the effective equation for \eqref{g8o173obui1bf} is
    \begin{equation}
        \ddot{y}=f_2(x^\omega,y)+\lambda^T C ,
    \end{equation}
    in the sense that $y^\omega\xrightarrow{F} y$ and $x^\omega\rightarrow x$. Notice that the convergence on $x$ is in the strong sense, i.e., $\lim_{\omega\rightarrow \infty} x^\omega(t)\rightarrow x(t)$ for all bounded $t>0$. This is because $x$ is purely slow, for which case, F-convergence implies strong convergence.

    Now consider a fully nonlinear $g(\cdot)$ with a non-degenerate Jacobian. Lemma \ref{Lemma1} gives that $g(y^\omega)=\mathcal{O}(1/\omega)$. Since $y^\omega$ is by assumption bounded, inverting $g$ leads to $y^\omega-y_0=\mathcal{O}(1/\omega)$. Consequently, the dynamics of $y^\omega$ approaches that of a forced oscillator (with equilibrium at $y_0$) at a $\mathcal{O}(1/\omega)$ timescale, because $g(\cdot)$ is approximated by its first order Taylor expansion:
    \begin{align*}
        \ddot{y^\omega}&=f_2(x^\omega,y^\omega)-\omega^2 g(y^\omega)^T \nabla g(y^\omega) \nonumber\\
        &=
        f_2(x^\omega,y^\omega)-\omega^2 (\nabla g(y_0) (y^\omega-y_0)^T+\mathcal{O}(\omega^{-2}))^T (\nabla g(y_0) + \text{Hess}\, g(y_0) (y^\omega-y_0)^T+\mathcal{O}(\omega^{-2})) \nonumber\\
        &=
        f_2(x^\omega,y^\omega)-\omega^2 (y^\omega-y_0) \nabla g(y_0)^T \nabla g(y_0) + \mathcal{O}(1),
    \end{align*}
    where nonlinearity $f_2+\mathcal{O}(1)$ again manifests as a slow force, which is dominated by the linear term that leads to asymptotically periodic oscillations. Hence, similar to the linear case, $y^\omega$ is locally ergodic on energy shell, the Lagrange multiplier $\lambda$ is well-defined, and the solution $x^\omega$, $y^\omega$ F-converges to the effective solution $x$, $y$.
\end{proof}

\begin{figure}[htb]
\vspace{-11pt}
\center
\includegraphics[width=0.8\textwidth]{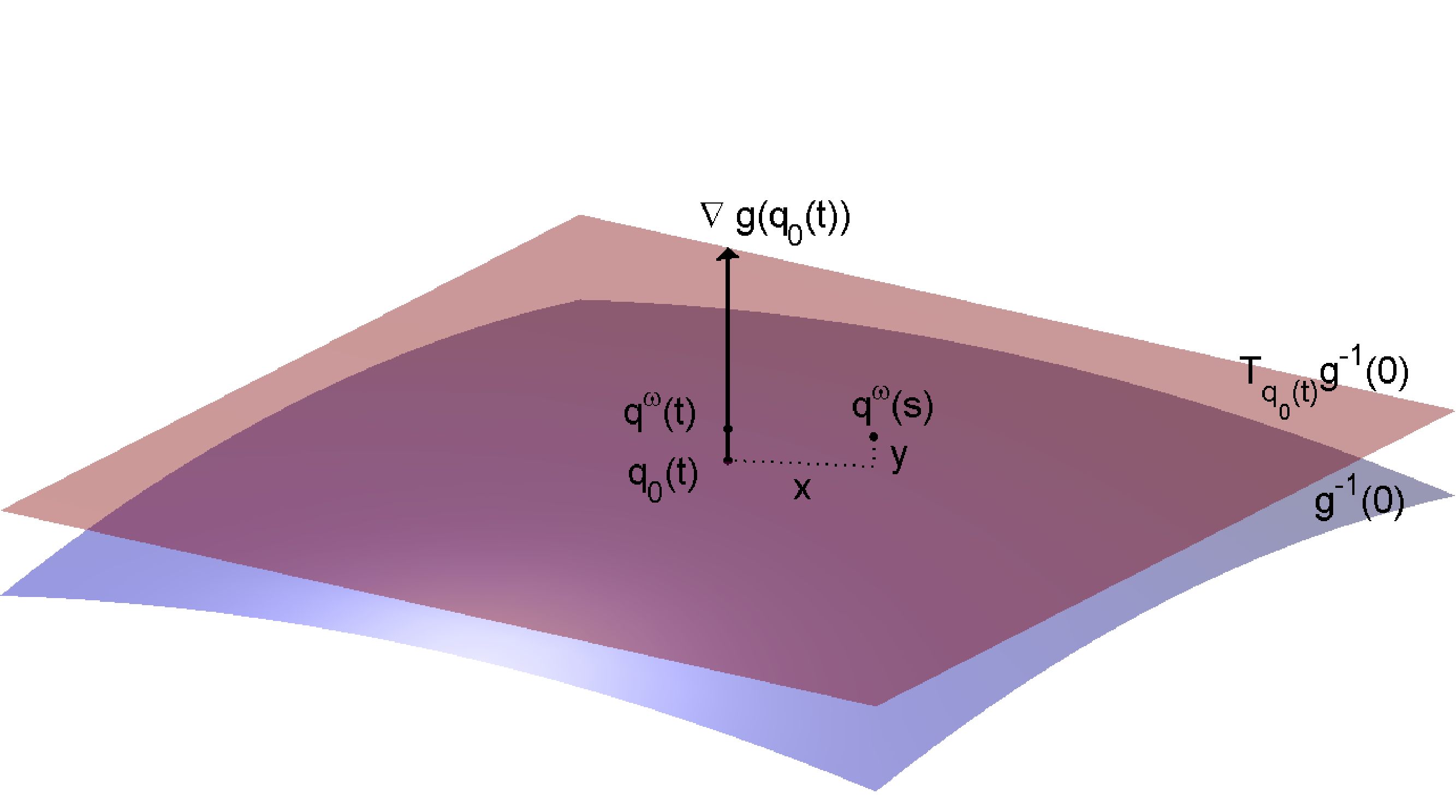}
\vspace{-15pt}
\caption{\footnotesize Multiscale geometry of penalized constrained dynamics -- $x$ and $y$ are slow and fast.}
\label{asymptoticLambdaTheoremFigure}
\end{figure}

\smallskip
\noindent \emph{Sketch of the proof of Theorem \ref{asymptoticLambdaTheorem}.} (Figure \ref{asymptoticLambdaTheoremFigure} illustrates the notations used in the proof to help understand the geometry.)
Since $g(q^\omega)$ is at most $\mathcal{O}(1/\omega)$ (Lemma \ref{Lemma1}), $q^\omega$ is close to the constraint manifold $g^{-1}(0)$ in the sense that if we define for all $t$:
\begin{equation}
    q_0(t):=\min_{q\in g^{-1}(0)} \|q-q^\omega(t)\|
\end{equation}
then $q^\omega(t)-q_0(t)=\mathcal{O}(1/\omega)$. Indeed, given that $\nabla g$ has the maximum rank, $\nabla g(q_0(t))$ spans the normal section (i.e., the subspace perpendicular to the tangent subspace) of the constraint manifold, in which $q^\omega(t)-q_0(t)$ also lies. Moreover, $g$ restricted to each normal section is an isomorphism, and both the restricted map and its inverse have bounded norms due to the boundedness of $q^\omega$ (i.e., compactness of the solution space) --- this is why $g(q^\omega)=\mathcal{O}(1/\omega)$ implies $q^\omega(t)-q_0(t)=\mathcal{O}(1/\omega)$.

The idea is that since $q^\omega$ is close enough, the constraint manifold can be locally viewed as a flat subspace, and F-convergence for this case has been proved in Lemma \ref{Lemma2}. More precisely, there exists a linear isomorphism $A_{q_0(t)}$, such that
\begin{equation}
    A_{q_0(t)} (q^\omega(t)-q_0(t))=\begin{bmatrix} 0 \\ y \end{bmatrix}
\end{equation}
where $y$ is a vector with codimension of $g^{-1}(0)$ and $0$ is a null vector.

For $q^\omega(s)$ with $s-t=\mathcal{O}(1/\omega)$, we will have a full-dimensional representation:
\begin{equation}
    A_{q_0(t)} (q^\omega(s)-q_0(t))=\begin{bmatrix} x \\ y \end{bmatrix},
\end{equation}
and $x$ and $y$ will respectively be the slow and fast variables, representing the constrained dynamics and fluctuations away from the constraint manifold (analogous to Lemma \ref{Lemma2}). This is because
\begin{align}
    \frac{d^2}{ds^2}\begin{bmatrix} x \\ y \end{bmatrix} &= A_{q_0(t)}(-\nabla V(q^\omega(s))-\omega^2 g(q^\omega(s))\nabla g(q^\omega(s))) \nonumber\\
    &= \begin{bmatrix} f_1(x,y) \\ f_2(x,y) \end{bmatrix} + \begin{bmatrix} \mathcal{O}(1) \\ -\omega^2 \tilde{g}(y)\nabla \tilde{g}(y) + \mathcal{O}(1) \end{bmatrix} ,
    \label{alsghqreogqhoo847}
\end{align}
where $f_1$ and $f_2$ are defined as $A_{q_0(t)}(-\nabla V(q^\omega(s))$. The $\mathcal{O}(1)$ in the 1st row of the right hand side of \eqref{alsghqreogqhoo847} is because $A_{q_0(t)}$ rotates the normal section to the y-direction, i.e.,
\begin{equation}
    A_{q_0(t)}\nabla g(q^\omega(s))=A_{q_0(t)}(\nabla g(q_0)+\mathcal{O}(1/\omega))=\begin{bmatrix} 0 \\ * \end{bmatrix} +\mathcal{O}(1/\omega) ,
\end{equation}
where $*$ is some non-zero expression, and certainly $\mathcal{O}(1/\omega)=\mathcal{O}(1)$.

The $\mathcal{O}(1)$ in the 2nd row of the right hand side of  \eqref{alsghqreogqhoo847} can also be intuitively obtained by using an analogous geometric argument, together with Taylor expansion.

Since \eqref{alsghqreogqhoo847} corresponds to the locally flat system \eqref{gh1orh3b1o4t143t1}, Lemma \eqref{Lemma2} proved the existence of an equivalent Lagrange multiplier as well as the F-convergence towards it. Moreover, \eqref{alsghqreogqhoo847} and the global dynamics near the curved constraint manifold \eqref{ourConstrainedDynamics} is linked via a coordinate transformation $q^\omega \mapsto A_{q_0}(q^\omega-q_0)$, which, naturally, is slowly varying as $q_0$ changes. Since averaging via F-convergence (Theorem 1.2 in \cite{FLAVOR10}) still works if the slow and fast variables are images of the original variable under a slowly varying diffeomorphism, the global dynamics \eqref{ourConstrainedDynamics} is F-convergent to a solution of \eqref{DAE2}. Notice that $g(q(t))=0$ in \eqref{DAE2} is automatically satisfied, because $\lim_{\omega\rightarrow\infty}g(q^\omega(t))=0$.

Finally, the solution to \eqref{DAE2} is also the solution to \eqref{DAE}. This is by the existence and uniqueness of the solution to differential algebraic equations with initial conditions.

\noindent
(Only main lines of the proof are provided; details are similar to analysis in \cite{FLAVOR10, SIM2}).
$\blacksquare$

\begin{Remark}
    The above proofs show that $g(q^\omega(s))$ is not only $\mathcal{O}(\omega^{-1})$ but $\mathcal{O}(\omega^{-2})$.
    \label{p8ohrfqilbvlorgbgouob14}
\end{Remark}

\end{document}